\documentclass[12pt]{amsart}
\usepackage{fullpage,url,amssymb,enumerate,colonequals}
\usepackage[all]{xy} 
\usepackage{mathrsfs} 

\usepackage[OT2,T1]{fontenc}
\DeclareSymbolFont{cyrletters}{OT2}{wncyr}{m}{n}
\DeclareMathSymbol{\Sha}{\mathalpha}{cyrletters}{"58}

\usepackage{color}

\newcommand{\defi}[1]{\textsf{#1}} 

\newcommand{\Aff}{\mathbb{A}}
\newcommand{\C}{\mathbb{C}}
\newcommand{\F}{\mathbb{F}}

\newcommand{\PP}{\mathbb{P}}
\newcommand{\Q}{\mathbb{Q}}
\newcommand{\R}{\mathbb{R}}
\newcommand{\Z}{\mathbb{Z}}
\newcommand{\Qbar}{{\overline{\Q}}}

\newcommand{\kbar}{{\overline{k}}}

\newcommand{\Fbar}{{\overline{\F}}}

\newcommand{\boldf}{\mathbf{f}}
\newcommand{\boldl}{\ensuremath{\boldsymbol\ell}}
\newcommand{\boldL}{\mathbf{L}}

\newcommand{\boldw}{\mathbf{w}}
\newcommand{\boldzero}{\mathbf{0}}

\newcommand{\calC}{\mathcal{C}}
\newcommand{\calD}{\mathcal{D}}
\newcommand{\calE}{\mathcal{E}}
\newcommand{\calF}{\mathcal{F}}

\newcommand{\calH}{\mathcal{H}}

\newcommand{\calJ}{\mathcal{J}}

\newcommand{\calM}{\mathcal{M}}

\newcommand{\calO}{\mathcal{O}}

\newcommand{\calS}{\mathcal{S}}
\newcommand{\calT}{\mathcal{T}}
\newcommand{\calU}{\mathcal{U}}

\newcommand{\CC}{\mathscr{C}}

\newcommand{\JJ}{\mathscr{J}}

\newcommand{\OO}{\mathscr{O}}
\newcommand{\WW}{\mathscr{W}}

\DeclareMathOperator{\Div}{Div}

\DeclareMathOperator{\Eq}{Eq}

\DeclareMathOperator{\Fix}{\tt Fix}
\DeclareMathOperator{\Frac}{Frac}

\DeclareMathOperator{\Gal}{Gal}

\DeclareMathOperator{\im}{im}

\DeclareMathOperator{\Log}{Log}
\DeclareMathOperator{\MakeDecentModel}{\tt MakeDecentModel}

\DeclareMathOperator{\NP}{NP}

\DeclareMathOperator{\Prob}{\bf P}
\DeclareMathOperator{\Proj}{Proj}

\DeclareMathOperator{\res}{res}

\DeclareMathOperator{\rk}{rk}

\DeclareMathOperator{\Sel}{Sel}

\DeclareMathOperator{\Sp}{Sp}
\DeclareMathOperator{\Spec}{Spec}

\newcommand{\good}{{\operatorname{good}}}

\newcommand{\smooth}{{\operatorname{smooth}}}

\newcommand{\tors}{{\operatorname{tors}}}

\newcommand{\E}{{\operatorname{\bf E}}}

\newcommand{\GL}{\operatorname{GL}}
\newcommand{\HH}{{\operatorname{H}}}

\newcommand{\injects}{\hookrightarrow}
\newcommand{\intersect}{\cap} 
\newcommand{\isom}{\simeq}

\newcommand{\surjects}{\twoheadrightarrow}
\newcommand{\tensor}{\otimes} 
\newcommand{\Union}{\bigcup} 

\newcommand{\Algorithm}{\textbf{Algorithm}\ }
\newcommand{\Subroutine}{\textbf{Subroutine}\ }

\newcommand{\rholog}{\rho \log}

\newcommand{\To}{\longrightarrow}

\numberwithin{equation}{section}

\newtheorem{theorem}[equation]{Theorem}
\newtheorem{lemma}[equation]{Lemma}
\newtheorem{corollary}[equation]{Corollary}
\newtheorem{proposition}[equation]{Proposition}

\theoremstyle{definition}
\newtheorem{definition}[equation]{Definition}

\newtheorem*{conjectureEq}{Conjecture~$\Eq_g(p)$}

\theoremstyle{remark}
\newtheorem{remark}[equation]{Remark}

\usepackage{color}

\definecolor{darkgreen}{rgb}{0,0.5,0}
\usepackage[
	colorlinks, citecolor=darkgreen,
	backref,
	pdfauthor={Bjorn Poonen, Michael Stoll}, 
	pdftitle={Most odd degree hyperelliptic curves have only one rational point},
]{hyperref}
\usepackage[alphabetic,backrefs,lite]{amsrefs} 

\begin{document}

\title{Most odd degree hyperelliptic curves \\ have only one rational point}
\subjclass[2010]{Primary 11G30; Secondary 14G25, 14G40, 14K15, 14K20}
\keywords{hyperelliptic curve, rational point, Chabauty's method, Selmer group}
\thanks{This article is to appear in \emph{Annals of Mathematics}.}

\author{Bjorn Poonen}
\address{Department of Mathematics, Massachusetts Institute of Technology, Cambridge, MA 02139-4307, USA}
\email{poonen@math.mit.edu}
\urladdr{http://math.mit.edu/~poonen/}

\author{Michael Stoll}
\address{Mathematisches Institut,
         Universit\"at Bayreuth,
         95440 Bayreuth, Germany.}
\email{Michael.Stoll@uni-bayreuth.de}
\urladdr{http://www.mathe2.uni-bayreuth.de/stoll/}

\dedicatory{In memory of Robert F. Coleman, who pioneered the effective approach to Chabauty's method}

\date{June 15, 2014}

\begin{abstract}
Consider the smooth projective models $C$ of curves $y^2=f(x)$
with $f(x) \in \Z[x]$ monic and separable of degree~$2g+1$.
We prove that for $g \ge 3$, a positive fraction of these have
only one rational point, the point at infinity.
We prove a lower bound on this fraction that
tends to $1$ as $g \to \infty$.
Finally, we show that $C(\Q)$ can be algorithmically computed 
for such a fraction of the curves.
The method can be summarized as follows: using $p$-adic analysis
and an idea of McCallum, 
we develop a reformulation of Chabauty's method
that shows that certain computable conditions imply $\#C(\Q)=1$;
on the other hand, using further $p$-adic analysis,
the theory of arithmetic surfaces, a new result 
on torsion points on hyperelliptic curves,
and crucially the Bhargava--Gross theorems on the average number and
equidistribution of nonzero $2$-Selmer group elements,
we prove that these conditions are often satisfied for $p=2$.
\end{abstract}

\maketitle


\section{Introduction}

In 1983, Faltings proved that if $C$ is a curve of genus $g>1$ over $\Q$,
then $C(\Q)$ is finite \cite{Faltings1983}.
Our goal is to study $C(\Q)$ as $C$ varies in a family,
namely the family $\calF_g$ of hyperelliptic curves 
$y^2=f(x)$ for $f(x) \in \Z[x]$ monic and separable of degree $2g+1$
for a fixed $g>1$.
Although we write an affine equation,
we mean the smooth projective model,
which has one point $\infty$ at infinity since $\deg f$ is odd.

Our main results are:
\begin{itemize}
\item For each $g \ge 3$, 
a positive fraction of the $C \in \calF_g$ satisfy $C(\Q)=\{\infty\}$
(Theorem~\ref{T:positive density}).
\item The fraction tends to~$1$ exponentially fast as $g \to \infty$
(Theorem~\ref{T:main}).
\item Chabauty's method \cites{Chabauty1941,Coleman1985chabauty}
at the prime $2$ is enough to yield
an effective algorithm to determine $C(\Q)$ for $C$ in a computable subset
whose density is positive for $g \ge 3$ 
and tends to~$1$ as $g \to \infty$ 
(Corollary~\ref{C:effective}).
\end{itemize}
(See Section~\ref{S:notation} for the precise definition of density.)
In particular, most monic integral polynomials of large odd degree 
never yield a square when evaluated on rational numbers.
This is the first time that Faltings' theorem has been made effective
for a positive fraction of a ``large'' family 
of curves with a rational point.
Note that the presence of a rational point makes it impossible to use
local methods to prove that $C(\Q)=\emptyset$,
and generally tends to make the determination of $C(\Q)$ more difficult.
On the other hand, the fraction is conjectured to be $1$ for all $g \ge 2$
(Remark~\ref{R:100 percent}).

Our proofs depend crucially on work of Bhargava and Gross 
on the average behavior of $2$-Selmer groups of hyperelliptic 
Jacobians~\cite{Bhargava-Gross-preprint}*{Theorems 11.1 and~12.4}.
Their work was inspired by the connection between pencils of quadrics
and hyperelliptic Jacobians
(work of Reid~\cite{Reid-thesis}, Donagi~\cite{Donagi1980}, 
and Wang~\cite{WangXJ-thesis}),
and by earlier work by Bhargava and Shankar 
for elliptic curves~\cite{Bhargava-Shankar-preprint1},
itself preceded by work of de Jong~\cite{DeJong2002} 
in the function field case
(see also \cite{Fouvry1993} and other references listed 
in \cite{Poonen2013-bourbaki}*{Section~2}).
Bhargava and Gross deduced corollaries for $C(\Q)$
from \cite{Bhargava-Gross-preprint}*{Theorem~11.1}:
specifically, they proved that 
for each $g \ge 2$,
there is a positive fraction of $C \in \calF_g$ 
satisfying $\# C(\Q) \le 3$, and 
for each $g \ge 3$,
the fraction of $C$ satisfying $\#C(\Q) < 20$ 
is more than $1/2$ \cite{Bhargava-Gross-preprint}*{Corollary~4}.
On the other hand, our arguments are essentially disjoint 
from those in \cite{Bhargava-Gross-preprint}:
we use \cite{Bhargava-Gross-preprint}*{Theorems 11.1 and~12.4}
only as a black box.

To explain how Bhargava and Gross passed from Selmer group information
to information on $C(\Q)$, 
and to explain why different 
ideas are needed to obtain \emph{our} results,
we must recall Chabauty's method
(see \cite{McCallum-Poonen2012} or~\cite{Stoll2006-chabauty} for a more detailed exposition).
Let $C$ be a curve over $\Q$ embedded in its Jacobian $J$.
In 1941, Chabauty~\cite{Chabauty1941},
inspired by an idea of Skolem~\cite{Skolem1934}, 
proved a weak form of what is now Faltings' theorem,
namely that if $\rk J(\Q) < g$, then $C(\Q)$ is finite.
Chabauty's approach was to bound $C(\Q)$
by $C(\Q_p) \intersect \overline{J(\Q)}$ inside $J(\Q_p)$,
where $\overline{J(\Q)}$ is the $p$-adic closure of $J(\Q)$ in $J(\Q_p)$.
Later, Coleman~\cite{Coleman1985chabauty} 
showed how to refine Chabauty's argument 
to obtain an \emph{explicit} upper bound on $\#C(\Q)$,
and improved bounds were given in~\cite{Stoll2006-chabauty}.
The latter bounds at odd primes of good reduction
were sufficient for Bhargava and Gross
to obtain their results for $\#C(\Q)$ above.

But it is impossible to reduce the Chabauty upper bound on $\#C(\Q)$ to $1$
if one knows only the size of the $2$-Selmer group
(when it does not force $J(\Q)$ to be finite),
because there is nothing to control 
the position of $\overline{J(\Q)}$ in $J(\Q_p)$.
Restricting the family of curves to a subfamily defined by
finitely many congruence conditions does not help: 
such conditions can determine the structure of $J(\Q_p)$,
but not the position of $\overline{J(\Q)}$ in $J(\Q_p)$, it seems.

To solve this problem,
we resurrect an idea of McCallum \cite{McCallum1994},
that knowledge of the $p$-Selmer group $\Sel_p J$ 
and its map to $J(\Q_p)/pJ(\Q_p)$
can sometimes be used to extract a tiny bit of information 
on the position of $\overline{J(\Q)}$ inside $J(\Q_p)$.
McCallum used this idea to study the arithmetic
of Fermat curves in 1994, but as far as we know, it has not been used since.
Fortunately, the method of Bhargava and Gross
yields not only the average size of $\Sel_2 J$, but also 
\emph{equidistribution} 
of the images of its nonzero elements
in $J(\Q_2)/2J(\Q_2)$ as $C$ varies 
(see Section~\ref{S:equidistribution} for the meaning of this),
even if one imposes finitely many congruence conditions on $C$.
We will prove that this suffices for the application of McCallum's idea:
we impose congruence conditions to control the position
of $C(\Q_2)$ in $J(\Q_2)$ for $C$ with good reduction at $2$,
and apply equidistribution to prove that for $g$ large enough (at least $3$),
at least a small positive fraction 
of these $C$ have the $2$-adic closure $\overline{J(\Q)}$ 
in a favorable position, i.e., intersecting $C(\Q_2)$ in only one point,
so that a $2$-adic Chabauty argument succeeds in proving $C(\Q)=\{\infty\}$.

To carry out the argument in the previous sentence,
we introduce a reformulation of Chabauty's method
in which we compute $C(\Q_2) \intersect \overline{J(\Q)}$ 
not in $J(\Q_2)$ directly but only after applying a sequence of maps
(see~\eqref{E:big}).
Specifically, 
we prove that if 
\[ \Sel_2 J \to V \colonequals \dfrac{J(\Q_2)}{2 J(\Q_2) + J(\Q_2)_{\tors}} \isom \F_2^g \]
is injective and the images of certain partially-defined maps
\[
	\rho\log \colon C(\Q_2) 
	\to J(\Q_2)
	\stackrel{\log}\surjects \Z_2^g 
	\dashrightarrow \PP^{g-1}(\Q_2) 
	\stackrel{\rho}\to \PP^{g-1}(\F_2) 
\]
and
\[
	\Sel_2 J \to V \isom \F_2^g \dashrightarrow \PP^{g-1}(\F_2)
\]
do not meet, then $C(\Q_2) \intersect \overline{J(\Q)}$ 
consists of torsion points of odd order
(Proposition~\ref{P:rational points are odd}).
Next, we exclude nontrivial torsion points 
by proving that 
most hyperelliptic curves do not contain any $\Q_2$-rational torsion points
except for Weierstrass points (Corollary~\ref{C:odd torsion on curves}): 
this follows from a new purely geometric statement, 
that the generic hyperelliptic curve contains no torsion points
except Weierstrass points 
(Theorem~\ref{T:torsion on generic hyperelliptic curve}).
Because Bhargava and Gross tell us how many nonzero Selmer elements there are
and how they are distributed in $V$,
it remains to show that $\rholog(C(\Q_2))$ is not too large.
When $g \ge 3$, this can be arranged for a positive fraction of curves
by $2$-adic congruence conditions on $C$
since it turns out that $\rholog(C(\Q_2))$ is locally constant as $C$
varies $2$-adically, if we exclude curves with unexpected torsion points
(Proposition~\ref{P:locally constant image of rholog}).

\begin{remark}
For $g=2$, 
it seems consistent with known results
that the conditions above on $C(\Q_2)$ and $\Sel_2 J$ fail for 100\% of curves:
the set $\PP^{g-1}(\F_2)$ is just too small.
But a \emph{3-adic} version of our argument would work even for $g=2$, 
if we knew equidistribution of nonzero $3$-Selmer group elements
(Remark~\ref{R:3-Selmer}).
\end{remark}

Now let us sketch how we strengthen the result
to obtain a fraction that tends to $1$ as $g \to \infty$.
We must continue to use the prime $2$,
because the equidistribution is 
currently known only for the $2$-Selmer group,
but now we must also consider $C$ with bad reduction at~$2$,
since the density of curves with good reduction at~$2$
tends to a number strictly less than $1$ as $g \to \infty$.
There are earlier studies of Chabauty's method in the bad reduction case,
such as \cite{Lorenzini-Tucker2002}*{Section~1}, 
\cite{McCallum-Poonen2012}*{Appendix}, 
and \cite{Katz-Zureick-Brown-preprint},
but the bounds they produce are not sharp enough for our purposes.
In fact, we must deal with curves with arbitrarily bad reduction at~$2$,
and our task is to prove that $\rholog(C(\Q_2))$ is usually small.
Since $\log|_{C(\Q_2)}$ is computed by integrating $1$-forms on residue disks,
which correspond to the $\F_2$-points in the smooth locus
$\calC^{\smooth}$ of the minimal proper regular model, 
it suffices to prove that
\begin{enumerate}[\upshape (1)]
\item\label{I:average number of smooth F2 points} 
the average size of $\calC^{\smooth}(\F_2)$ is small, and
\item\label{I:image of sigma_C on a residue disk}
the image of $\rholog$ on each residue disk is small.
\end{enumerate}

As for~\eqref{I:average number of smooth F2 points},
the Ogg--Saito formula~\cite{Saito1988} together 
with \cite{Liu1994}*{Proposition~1} 
bounds the number $c$ of connected components of $\calC_{\F_2}^{\smooth}$
in terms of the Deligne discriminant of $\calC$, 
but this discriminant is hard to compute,
and it is not known whether $c$ can be bounded 
in terms of the usual discriminant of the polynomial $f(x)$: 
the best results in this direction we know are those 
in~\cite{Liu1996}*{Section~9}.
So instead we analyze the random variable $\#\calC^{\smooth}(\F_2)$ 
by explicitly blowing up non-regular $\F_2$-points
until we have an approximation to $\calC$;
this leads to a recursive analysis of a 
Bienaym\'e--Galton--Watson-like process (Lemma~\ref{L:E X_n}).
The result (Theorem~\ref{T:average number of smooth points}) 
is that the average of $\#\calC^{\smooth}(\F_2)$
is at most $3$.

As for~\eqref{I:image of sigma_C on a residue disk},
this amounts to bounding the image of an analytic curve
in $\PP^{g-1}(\Q_2)$ under the reduction map to $\PP^{g-1}(\F_2)$.
We do not know of results in the literature on this kind of 
nonarchimedean analysis problem, 
though it is reminiscent of tropical geometry.
To handle it, we apply the $p$-adic Weierstrass preparation theorem
to replace the power series defining the analytic curve
by polynomials,
and hence reduce to the case of an algebraic rational curve.
We reinterpret the map $\PP^1_{\Q_2} \to \PP^{g-1}_{\Q_2}$ defining this curve
as a rational map $\PP^1_{\Z_2} \dashrightarrow \PP^{g-1}_{\Z_2}$,
whose indeterminacy at $\F_2$-points we resolve,
the upshot being that our image can be bounded 
in terms of the complexity of a tree of rational curves over $\F_2$
(see Section~\ref{S:curves under reduction}).

Many of our arguments work also at primes $p$ other than $2$,
so we work in this more general context when possible.
On the other hand, the Bhargava--Gross equidistribution theorem
is known only for $2$-Selmer elements, so the final results for higher $p$
must remain conditional for the time being.

\begin{remark}
Independently of the present paper,
Bhargava \cite{Bhargava-most-preprint} has proved that in the family of all
not-necessarily-monic \emph{even}-degree genus $g$ hyperelliptic curves
over $\Q$,
the density of those that have no rational points
is $1 - o(2^{-g})$ as $g \to \infty$.
Although the statement is similar to that of our Theorem~\ref{T:main},
and relies on average behavior of $2$-Selmer elements,
his proof is otherwise completely different:
it does not need an equidistribution theorem or Chabauty's method,
because there are more methods available for proving the nonexistence
of rational points than for determining the rational points when one exists.
Specifically, his proof controls the average size of the 
``fake $2$-Selmer set'' of the curve,
in effect showing that for most curves, 
all the relevant finite \'etale covers fail to have local points.
\end{remark}

\begin{remark}
For the family of \emph{monic} even-degree 
genus $g$ hyperelliptic curves over $\Q$, 
Shankar and Wang \cite{Shankar-Wang-preprint} 
have adapted the method of \cite{Bhargava-Gross-preprint}
to prove analogous theorems on the average size and equidistribution
of $2$-Selmer groups,
and then have adapted the method of the present paper
to prove that the density of such curves 
that have only the two rational points at infinity
is at least $1 - (48g+120) 2^{-g}$.
Just as the presence of one rational point makes it more difficult
to control the set of all rational points,
the presence of two rational points 
means that their argument must be more complicated than ours
in certain places.
\end{remark}


\section{Notation}
\label{S:notation}

For any field $k$, let $\kbar$ be an algebraic closure.

We fix a prime $p$. 
As usual, we define $\Z_p \colonequals \varprojlim \Z/p^n\Z$ and
its fraction field $\Q_p \colonequals \Frac \Z_p$.
More generally for any place $v$ of $\Q$, 
let $\Q_v$ be the completion of $\Q$ at $v$.
Let $\C_p$ be the completion of $\Qbar_p$,
let $\calO_{\C_p}$ be its valuation ring,
and let $D_1$ be the open unit disk in $\C_p$.
Let $v_p$ be the $p$-adic valuation on $\C_p$, normalized so that $v_p(p)=1$.

We fix $g \in \Z_{\ge 1}$.
For a field $k$,
let $\PP$ be the usual map $k^g \setminus \{\boldzero\} \to \PP^{g-1}(k)$.
We write $\rho$ for the reduction map 
$\PP^{g-1}(\Q_p) = \PP^{g-1}(\Z_p) \to \PP^{g-1}(\F_p)$
or
for the composition
$\Q_p^g \setminus \{\boldzero\} \stackrel{\PP}\to \PP^{g-1}(\Q_p) \stackrel{\rho}\to \PP^{g-1}(\F_p)$.
If $T$ is a subset of a set $S$,
and $f$ is a function defined only on $T$, 
then $f(S)$ means $f(T)$;
for example, we may write $\rho(\Q_p^g)=\PP^{g-1}(\F_p)$.

A \defi{variety} is a separated scheme $X$ of finite type over a field,
and $X$ is called \defi{nice} if it is smooth, projective, and geometrically
integral.
If $X$ and $T$ are $S$-schemes, define $X_T \colonequals X \times_S T$;
in this context we sometimes write $R$ instead of $\Spec R$.
Given a ring $R$, and $f \in R[x]$ of degree $2g+1$,
by the \defi{standard compactification} of $y^2=f(x)$ we mean the $R$-scheme
$\Proj R[x,y,z]/(y^2 - z^{2g+2} f(x/z))$ 
where $\deg x = \deg z =1$ and $\deg y = g+1$;
it can be covered by two affine patches, 
one isomorphic to $y^2=f(x)$ and the other to $y^2 = x^{2g+2} f(1/x)$.
For any domain $R$ of characteristic not~$2$,
denote by $\calF_g(R)$ the set of all nice genus~$g$ curves (over $\Frac R$) 
arising as the standard compactification of 
\[
   y^2 = x^{2g+1} + a_{1} x^{2g} + a_{2} x^{2g-1} + \cdots + a_{2g} x + a_{2g+1} 
\]
for some $a_1, a_2, \ldots, a_{2g+1} \in R$. 
This can be identified with $R^{2g+1}$
minus the zero set of the discriminant of the polynomial in~$x$ in the equation
above. We set $\calF_g \colonequals \calF_g(\Z)$.

Essentially following \cite{Bhargava-Gross-preprint},
for $C \in \calF_g$ corresponding to $(a_1, \ldots, a_{2g+1})$, 
define the \defi{height} of~$C$ as
\[
	H(C) \colonequals \max\bigl\{|a_1|, |a_2|^{1/2}, \ldots, 
			|a_{2g}|^{1/2g}, |a_{2g+1}|^{1/(2g+1)}\bigr\} . 
\]
(Actually, \cite{Bhargava-Gross-preprint} required $a_1=0$,
but this makes little difference: see Remark~\ref{R:a_1=0}.
Also, their height is the $2g(2g+1)$-th power of what we have written.)
We set 
\[ \calF_{g,X} \colonequals \{C \in \calF_g : H(C) < X\}. \]
The \defi{density} of a subset $S \subseteq \calF_g$ 
is 
\[ \mu(S) \colonequals \lim_{X \to \infty} \#(S \cap \calF_{g,X})/\#\calF_{g,X}, \]
if the limit exists.
Define \defi{lower density} and \defi{upper density}
by replacing $\lim$ by $\liminf$ or $\limsup$, respectively.
If $S \subseteq T \subseteq \calF_g$ and $\mu(T)>0$,
the \defi{relative density} of $S$ in $T$ is $\mu(S)/\mu(T)$, 
if $\mu(S)$ exists;
similarly define \defi{relative lower density} 
and \defi{relative upper density}. 
If $f \colon \calF_g \to \R$ is a function, 
then we say that $f$ has \defi{average}~$\alpha$
on~$\calF_g$ if
\[ 
	\lim_{X \to \infty} \frac{\sum_{C \in \calF_{g,X}} f(C)}{\#\calF_{g,X}} 
	= \alpha . 
\]
We say that the average of~$f$ is at most~$\alpha$
if the $\limsup$ is at most $\alpha$.
Similarly define the average of a function on an infinite subset of $\calF_g$.

Restrict the normalized Haar measure on $\Z_p^{2g+1}$
to obtain a probability measure on $\calF_g(\Z_p)$.
Let $\Prob(S)$ be the probability of an event $S$.
For a random variable $X$ defined on $\calF_g(\Z_p)$,
let $\E X$ denote its average.
If $X$ is a random variable defined only on a positive-measure
subset $S$ of $\calF_g(\Z_p)$,
then $\E X$ denotes the average of $X$ conditioned on the event $S$.


\section{Images of curves under reduction}
\label{S:curves under reduction}

\subsection{Algebraic curves}

By the \defi{degree} of a morphism $\phi \colon C \to \PP^{g-1}$,
where $C$ is a nice curve,
we mean $\deg \phi^* \OO(1)$.
If $f_1,\ldots,f_g$ are single-variable polynomials of degree at most $n$,
not all zero,
then the rational map $(f_1:\cdots:f_g) \colon \Aff^1 \dashrightarrow \PP^{g-1}$
extends to a morphism $\PP^1 \to \PP^{g-1}$ of degree at most $n$.

\begin{proposition} \label{P:s&r-image}
Let $\phi \colon \PP^1_{\Q_p} \to \PP^{g-1}_{\Q_p}$ be of degree~$n$.
Then $\#\rho\bigl(\phi(\PP^1(\Q_p))\bigr) \le np+1$.
\end{proposition}

If $\phi$ has good reduction (i.e., extends to a morphism
$\PP^1_{\Z_p} \to \PP^{g-1}_{\Z_p}$),
then $\rho(\phi(\PP^1(\Q_p)))$ is contained in $\phi(\PP^1(\F_p))$,
which has size at most $p+1$.
In the general case, we will blow up $\PP^1_{\Z_p}$ 
to resolve the indeterminacy,
and control the resulting increase in the number of $\F_p$-points 
on the source.

\begin{proof}[Proof of Proposition~\ref{P:s&r-image}]
After iteratively blowing up $\F_p$-points on $\PP^1_{\Z_p}$,
we obtain a proper regular $\Z_p$-scheme $S'$ such that $\phi$
extends to a rational map $\phi' \colon S' \dashrightarrow \PP^{g-1}_{\Z_p}$ 
defined at all $\F_p$-points of $S'$.
If we also blow up closed points of higher degree, 
we obtain $S''$ such that $\phi$ extends to 
a \emph{morphism} $\phi'' \colon S'' \to \PP^{g-1}_{\Z_p}$, 
as in \cite{Lichtenbaum1968}*{II.D, Proposition~4.2}.
Then $\PP^1(\Q_p) = S''(\Z_p) = S'(\Z_p)$, 
so $\rho(\phi(\PP^1(\Q_p))) = \rho(\phi'(S'(\Z_p))) \subseteq \phi'(S'(\F_p))$.

The construction of $S'$ shows that $S'_{\F_p}$ is a 
strict normal crossings divisor whose components are copies
of $\PP^1_{\F_p}$ meeting at $\F_p$-points.
Name these components $S'_1,S'_2,\ldots$
in the order that they were produced by the blowing up, 
starting with $S'_1$ being the strict transform of $\PP^1_{\F_p}$.
Since $\deg {\phi''_{\F_p}}^* \OO(1) = \deg \phi^* \OO(1) = n$,
the morphism $\phi''$ is non-constant on at most $n$ components
of $S''_{\Fbar_p}$,
so $\phi'$ is non-constant on at most $n$ of the sets $S'_i(\F_p)$.
Let $\Sigma_0=\{P\}$ for some $P \in S'_1(\F_p)$,
and for $i \ge 1$, let $\Sigma_i \colonequals \Union_{j=1}^i S'_j(\F_p)$.
We have $\# \phi'(\Sigma_0) = 1$.
Incrementing $i$ increases $\# \phi'(\Sigma_i)$ by $0$
if $\phi'$ is constant on $S'_{i+1}(\F_p)$
and by at most $p$ otherwise,
since one of the $p+1$ points of $S'_{i+1}(\F_p)$ was already in $\Sigma_i$.
Thus $\# \phi'(\Sigma_i)$ increases at most $n$ times,
by at most $p$ each time, starting from $1$.
Hence $\#\phi'(S'(\F_p)) \le np+1$.
\end{proof}

\begin{remark}
For each $(n,p)$, the bound in Proposition~\ref{P:s&r-image} is sharp:
for any $g>n$, define $\phi$ by taking $f_j(t) \colonequals p^{j(j-1)} t^{j-1}$
for $1 \le j \le n+1$, and $f_j(t) \colonequals 0$ for $n+1 < j \le g$;
then $\rho(\phi(\PP^1(\Q_p)))$ is a chain of $n$~lines in~$\PP^{g-1}(\F_p)$.
\end{remark}

\begin{remark}
The proof of Proposition~\ref{P:s&r-image} suggests a down-to-earth
algorithm for computing $\rho(\phi(\PP^1(\Q_p)))$.
Namely, one iteratively subdivides $\PP^1(\Q_p)$ into disks 
until $\rho \circ \phi$ is constant on each.
\end{remark}

\subsection{Analytic curves}

Recall the notation $\Z_p$, $\Q_p$, $\C_p$, and $D_1$ 
from Section~\ref{S:notation}.

\begin{definition} \label{D:nr}
Let $\boldw = (w_1, w_2, \ldots, w_g) \in \C_p[\![t]\!]^g - \{\boldzero\}$.
Let $w_{j,n}$ be the coefficient of~$t^n$ in~$w_j$. 
Define the \defi{Newton polygon} $\NP(\boldw)$ as the
lower convex hull of the set
\[
	\left\{ \bigl(n, v_p(w_{j,n})\bigr) : 1 \le j \le g, \; n \ge 0 \right\}.
\]
Suppose that the minimum of the $y$-coordinates 
of the vertices of $\NP(\boldw)$ is attained;
then for the vertices attaining this minimum,
let $n_{\boldw}$ and $N_{\boldw}$ be the minimum and maximum
$x$-coordinates if they exist (the maximum might not exist).
When $\boldw$ consists of a single $w$, we also write $\NP(w)$, $n_w$, $N_w$.
\end{definition}

\begin{remark}
\label{R:span}
The Newton polygon $\NP(\boldw)$ depends only on the
$\Z_p$-span of the $w_i$.
\end{remark}

\begin{remark}
\label{R:linear combination}
Let $\boldw \in \Q_p[\![t]\!]^g - \{\boldzero\}$.
Let $R$ be the valuation ring of an unramified extension of $\Q_p$.
If $\lambda_1,\ldots,\lambda_g \in R$ have $\F_p$-independent images in $R/pR$,
then
\[ \NP(\boldw) = \NP(\lambda_1 w_1 + \cdots + \lambda_g w_g) . \]
\end{remark}

\begin{remark}
\label{R:NP and zeros}
For $w \in \C_p[\![t]\!] - \{0\}$ for which $n_w$ is defined,
the theory of Newton polygons 
\cite{Koblitz1984}*{Corollary on p.~106} 
implies that $n_w = \#\left(\textup{zeros of $w$ on $D_1$}\right)$; 
we write $\#(\;)$ instead of $\#\{\;\}$ to indicate that
we are counting zeros with multiplicity.
\end{remark}

Following \cite{Stoll2006-chabauty}*{Section~6},
for $n \ge 0$, define
\[ \delta(p,n) \colonequals \max\{d \ge 0 : v_p(n+1) + d \le v_p(n+d+1)\} \,. \]

\begin{proposition} \label{P:integral}
Suppose that $\boldw \in \Q_p[\![t]\!]^g \setminus \{\boldzero\}$
has $p$-adically bounded coefficients.
Let $\boldl \in \Q_p[\![t]\!]^g$ be such that $d\boldl/dt = \boldw$.
Then $\# \rho(\boldl(p\Z_p)) 
	\le p \bigl(n_{\boldw} + 1 + \delta(p, n_{\boldw}) \bigr) + 1$.
\end{proposition}

\begin{proof}
Let $\boldL(t) \colonequals \boldl(pt)$.
The lattice point responsible for the value of $N_{\boldL}$ 
is a vertex of $\NP(L_j)$ for some $j$.
We ensure that it is a vertex of $\NP(L_i)$ for \emph{every} $i$
by adding $w_j$ to each $w_i$ for which this does not yet hold
(this has the effect of applying a linear change of variable
to the codomain $\PP^{g-1}(\F_p)$ of $\rho \circ \boldl$, 
but does not change the size of the image).
By the $p$-adic Weierstrass preparation theorem
(see \cite{Koblitz1984}*{p.~105, Theorem~14} and its proof),
$L_i = f_i u_i$
for some $f_i \in \Q_p[t]$ of degree $N_{\boldL}$ 
and $u_i \in 1 + p t \Z_p[\![t]\!]$ converging on~$\Z_p$.
Set $\boldf = (f_1, \ldots, f_g)$.
For $\tau \in \Z_p$, we have $u_i(\tau) \in 1 + p\Z_p$,
so $\rho(\boldL(\tau)) = \rho(\boldf(\tau))$.
Hence $\rho(\boldl(p\Z_p)) 
	= \rho(\boldL(\Z_p)) 
	= \rho(\boldf(\Z_p)),$
which has size at most $p N_{\boldL} + 1$
by Proposition~\ref{P:s&r-image} applied to the morphism $\PP^1 \to \PP^{g-1}$
defined by $\boldf$.
Finally,  
we prove $N_{\boldL} \le n_{\boldw} + 1 + \delta(p, n_{\boldw})$:
by Remark~\ref{R:linear combination}, we may reduce to the case $g=1$,
with coefficients now in an unramified extension; 
this case can be deduced easily by considering the slopes of the
(now standard) Newton polygon; 
cf.~\cite{Stoll2006-chabauty}*{Proposition~6.3}.
\end{proof}


\section{The logarithm map}
\label{S:logarithm}

Let $J$ be an abelian variety over $\Q_p$.
Let $T_0 J$ be the tangent space to $J$ at $0$.
Integrating $1$-forms defines an analytic group homomorphism
$\log \colon J(\Q_p) \to T_0 J \isom \Q_p^g$
whose kernel is the torsion subgroup $J(\Q_p)_{\tors}$.
Since $\log$ is a local diffeomorphism and $J(\Q_p)$ is compact,
its kernel $J(\Q_p)_{\tors}$ is finite,
and its image will be $\Z_p^g$
for a suitable choice of identification $T_0 J \isom \Q_p^g$;
this identification corresponds to 
a choice of basis $\omega_1,\ldots,\omega_g$ of $\HH^0(J,\Omega^1)$,
which we now fix.
Any commutative extension of $\Z_p^g$ by a finite abelian group
is split, even as a topological group,
so $J(\Q_p) \isom \Z_p^g \times J(\Q_p)_{\tors}$.
Let $\rholog$ be the composition
\[
	J(\Q_p) \stackrel{\log}\surjects \Z_p^g 
	\stackrel{\rho}\dashrightarrow \PP^{g-1}(\F_p),
\]
defined on $J(\Q_p) \setminus J(\Q_p)_{\tors}$.


\section{Image of the curve under the logarithm map}

For this section, let $C$ be a nice curve of genus $g \ge 1$ over $\Q_p$
with a $\Q_p$-point $\infty$.
Embed $C$ in its Jacobian $J$ by sending $\infty$ to $0$.
Define $\log$ as in Section~\ref{S:logarithm}.
Let $\calC \to \Spec \Z_p$ be the minimal proper regular model of $C$.

\subsection{Image of one residue disk}

\begin{definition}
A \defi{residue disk} $D \subseteq C(\C_p)$ is
the preimage of a point $P \in \calC^{\smooth}(\F_p)$
under $C(\C_p) = \calC(\calO_{\C_p}) \to \calC(\Fbar_p)$;
then let $D(\Q_p) \colonequals D \intersect C(\Q_p)$.
A \defi{uniformizer} for $D$
is a regular function $t$ on an open neighborhood of $P$ in $\calC$
such that $t$ reduces to a uniformizer at $P$ on $\calC_{\F_p}$.
\end{definition}

Every point of $C(\Q_p)$ reduces to a point of $\calC^{\smooth}(\F_p)$,
so $C(\Q_p)$ is the disjoint union of the open sets $D(\Q_p)$.

Let $D$ be a residue disk with uniformizer $t$.
Then $t$ defines a diffeomorphism $D \to D_1$
identifying $D(\Q_p)$ with $p\Z_p$.
The restriction of any $\omega \in \HH^0(C,\Omega^1)$ to $D$
corresponds to an analytic $1$-form $w(t)\,dt$ on $D_1$, 
for some $w \in \Q_p[\![t]\!]$ with bounded coefficients.
Applying this to $\omega_i|_C$ defines some $w_i$.
Let $\boldw \colonequals (w_1,\ldots,w_g)$
and $n_D \colonequals n_{\boldw}$.

\begin{proposition} \label{P:sigmaC_on_residue_disk}
We have 
$\# \rholog(D(\Q_p)) \le p \bigl(n_D + 1 + \delta(p, n_D)\bigr) + 1.$
\end{proposition}

\begin{proof}
Since $\log$ is defined by integrating $(\omega_1,\ldots,\omega_g)$,
the composition
\[
\xymatrix{
	p\Z_p \ar[r]^-{\sim} 
	& D(\Q_p) \ar@{^{(}->}[r]
	& C(\Q_p) \ar@{^{(}->}[r]
	& J(\Q_p) \ar[r]^-{\log} 
	& \Q_p^g \\
}
\]
is some $\boldl \in \Q_p[\![t]\!]^g$ with $d\boldl/dt = \boldw$.
Then
\[
	\# \rholog(D(\Q_p)) 
	= \# \rho(\boldl(p\Z_p))
	\le p \bigl(n_D + 1 + \delta(p, n_D)\bigr) + 1
\]
by Proposition~\ref{P:integral}.
\end{proof}

\subsection{Image of many residue disks}

\begin{lemma}
\label{L:sum of n_D}
Let $\calD$ be the set of residue disks on $C$.
Then $\sum_{D \in \calD} n_D \le 2g-2$.
\end{lemma}

\begin{proof}
Let $\lambda_1,\ldots,\lambda_g$ be as in Remark~\ref{R:linear combination}.
Let $\omega = \sum_{i=1}^g \lambda_i \omega_i|_C \in \HH^0(C_{\C_p},\Omega^1)$.
On a residue disk $D$ with uniformizer $t$, 
express $\omega$ as $w(t)\,dt$,
so $w  = \sum_{i=1}^g \lambda_i w_i \in \C_p[\![t]\!]$;
then 
\[
	n_D = n_{\boldw} = n_w = \#\left(\textup{zeros of $w$ on $D_1$}\right)
	= \#\left(\textup{zeros of $\omega$ on $D$}\right),
\]
by Remark~\ref{R:NP and zeros}.
Thus
\[
	\sum_{D \in \calD} n_D \le 
	\#\left(\textup{zeros of $\omega$ on $C(\C_p)$}\right)
	= 2g-2.\qedhere
\]
\end{proof}

\begin{proposition} 
\label{P:image of general C}
Let $d \colonequals \#\calC^{\smooth}(\F_p)$.
Then 
\[
	\#\rholog(C(\Q_p)) \le 
	\begin{cases}
		5d + 6g - 6, &\textup{if $p=2$,} \\
		(p+1)d + \dfrac{p^2-p}{p-2}(2g-2), &\textup{if $p>2$.} \\
	\end{cases}
\]
\end{proposition}

\begin{proof}
Sum the bound of Proposition~\ref{P:sigmaC_on_residue_disk} over all $D$
and use Lemma~\ref{L:sum of n_D} to obtain
\[
	\#\rholog(C(\Q_p)) \le 
	p(2g-2) + (p+1)d + p \sum_D \delta(p,n_D).
\]
If $p=2$, use the bound $\delta(2,n) \le 1 + n/2$ 
(and use Lemma~\ref{L:sum of n_D} again) to conclude.
If $p>2$, then $\sum_D \delta(p,n_D) \le \Delta_p(d,2g-2)$,
where
\[
	\Delta_p(d,N) \colonequals 
	\max\Bigl\{ \sum_{j=1}^d \delta(p, n_j) : 
		n_j \in \Z_{\ge 0} \textup{ and } \sum_{j=1}^d n_j \le N \Bigr\};
\]
use the bound $\Delta_p(d,N) \le N/(p-2)$ 
of \cite{Stoll2006-chabauty}*{Lemma~6.2}.
\end{proof}


\section{Image of the rational points under the logarithm map}

We now assume that $C$ is a nice curve of genus $g \ge 1$ over~$\Q$
with a $\Q$-point $\infty$, which we use as base-point for embedding $C$
in its Jacobian~$J$.
Taking Galois cohomology of 
$0 \to J[p] \to J \stackrel{p}{\to} J \to 0$
over $\Q$ and over $\Q_v$ for all places $v$ of $\Q$
yields the rows in the following commutative diagram,
where $\delta$ now denotes a connecting homomorphism: 
\[
\begin{split}
  \xymatrix{ \dfrac{J(\Q)}{p J(\Q)} \ar[d] \ar@{^{(}->}[r]^-{\delta}
                       & \HH^1(\Q, J[p]) \ar[d]^{\res} \\
              \displaystyle\prod_v \dfrac{J(\Q_v)}{p J(\Q_v)} \ar@{^{(}->}[r]^-{\delta'}
                       & \displaystyle\prod_v \HH^1(\Q_v, J[p]). \\
            }
\end{split}
\]
The \defi{$p$-Selmer group} of $J$ is defined by
\[
	\Sel_p J \colonequals \{\, \xi \in \HH^1(\Q,J[p]) : \res(\xi) \in \im(\delta') \,\}.
\]
In particular, $\res$ restricts to a homomorphism
$\Sel_p J \to J(\Q_p)/p J(\Q_p)$.
Since the $p$-adic closure $\overline{J(\Q)}$ 
contains a finite-index closed subgroup
that is a free $\Z_p$-module of finite rank,
the natural map 
$\mu \colon J(\Q)/p J(\Q) \to \overline{J(\Q)}/p \overline{J(\Q)}$ 
is surjective.
Choose $\log \colon J(\Q_p) \surjects \Z_p^g$ as in Section~\ref{S:logarithm}.
Then we have a diagram
\begin{equation}
\begin{split}
\label{E:big}
\xymatrix{
C(\Q) \ar@{^{(}->}[rr] \ar@{^{(}->}[d] && C(\Q_p) \ar@{^{(}->}[d] \\
J(\Q) \ar@{^{(}->}[r] \ar@{->>}[d] & \overline{J(\Q)} \ar@{^{(}->}[r] \ar@{->>}[d] & J(\Q_p) \ar@{->>}[r]^-{\log} \ar@{->>}[d] \ar@/^5pc/@{-->}[rrd]^-{\rholog} & \Z_p^g \ar@{-->}[rd]^-{\rho} \ar@{->>}[d] \\
\dfrac{J(\Q)}{pJ(\Q)} \ar@{->>}[r]^{\mu} \ar@{^{(}->}[rd]_-{\delta} & \dfrac{\overline{J(\Q)}}{p\overline{J(\Q)}} \ar[r] & \dfrac{J(\Q_p)}{pJ(\Q_p)} \ar@{->>}[r]^-{\log \tensor \F_p} & \F_p^g \ar@{-->}[r]^-{\PP} & \PP^{g-1}(\F_p) \\
& \Sel_p J \ar[ru] \ar@/_/[rru]_-{\sigma} \ar@(r,d)@{-->}[rrru]_-{\PP\sigma} \\
}
\end{split}
\end{equation}
in which $\sigma$ and $\PP\sigma$ are defined as the compositions,
so that the diagram commutes on elements for which the maps are defined.

Let $J(\Q_p)[p']$ be the set of points of finite order prime to $p$ 
in~$J(\Q_p)_{\tors}$. 
By Section~\ref{S:logarithm}, $J(\Q_p) \isom \Z_p^g \times F$
for a finite abelian group $F$,
so $J(\Q_p)[p']$ is the set of points that are
infinitely $p$-divisible in~$J(\Q_p)$.

\begin{proposition}
\label{P:rational points are odd}
If $\sigma$ is injective 
and the images $\rholog(C(\Q_p))$ and $\PP\sigma(\Sel_p J)$ are disjoint,
then $C(\Q_p) \intersect \overline{J(\Q)} \subseteq J(\Q_p)[p']$.
\end{proposition}

\begin{proof}
Suppose that the hypotheses hold.
Then $\sigma \delta$ is injective,
and \eqref{E:big} shows that the surjection $\mu$ is an isomorphism 
and that $\overline{J(\Q)}/p\overline{J(\Q)} \to \F_p^g$
is injective.

Suppose also that the conclusion fails;
fix $P \in C(\Q_p) \intersect \overline{J(\Q)}$ not in~$J(\Q_p)[p']$. 
Then $P$ is not infinitely $p$-divisible in $\overline{J(\Q)}$,
so $P = p^n Q$ for some $n \ge 0$
and some $Q \in \overline{J(\Q)}$ outside $p \overline{J(\Q)}$.
Let $\bar{Q}$ be the image of $Q$ in $\overline{J(\Q)}/p\overline{J(\Q)}$.
Then $\bar{Q}$ is nonzero, so its image in $\F_p^g$ is nonzero, 
and $\PP\sigma\bigl(\delta\mu^{-1}(\bar{Q})\bigr)$ is defined. 
Tracing through \eqref{E:big}, we have
\[ \PP\sigma(\Sel_p J) \ni \PP\sigma\bigl(\delta\mu^{-1}(\bar{Q})\bigr)
     = \rholog(Q) = \rholog(p^n Q) = \rholog(P) \in \rholog(C(\Q_p)),
\]
contradicting the assumption that 
$\rholog(C(\Q_p))$ and $\PP\sigma(\Sel_p J)$ are disjoint.
\end{proof}

We state the following consequence explicitly, since it may have applications
outside the context of this paper.

\begin{corollary}
\label{C:useful outside the paper}
  Let $C$ be a nice curve of genus~$g \ge 1$ over~$\Q$ with a rational point
  $\infty \in C(\Q)$. We embed $C$ in its Jacobian~$J$ using $\infty$ as base-point.
  Let $p$ be a prime number such that in diagram~\eqref{E:big} $\sigma$ is injective
  and the images $\rholog(C(\Q_p))$ and $\PP\sigma(\Sel_p J)$ are disjoint.
  Then $C(\Q) \subseteq J(\Q)[p']$.
\end{corollary}

\begin{proof}
  Apply Proposition~\ref{P:rational points are odd} and use that
  $C(\Q) = C(\Q_p) \intersect J(\Q) \subseteq C(\Q_p) \intersect \overline{J(\Q)}$.
\end{proof}

Given $C$, $\infty$, and a prime $p$,
the hypotheses on $p$ in Corollary~\ref{C:useful outside the paper}
can be checked explicitly:
\begin{itemize}
  \item We can compute a regular model~$\calC$ of~$C$ over~$\Z_p$, which
        gives us a covering of~$C(\Q_p)$ by residue disks.
  \item We can compute a basis of the space of regular 1-forms on~$C$
        and their integrals on a set of representatives of generators
        of $J(\Q_p)/(p J(\Q_p) + J(\Q_p)_\tors)$. This gives us
        $\log \colon J(\Q_p) \surjects \Z_p^g$.
  \item We can then compute $\rholog(C(\Q_p))$ by evaluating $\rholog$
        on each residue disk.
  \item The $p$-Selmer group of~$J$ can be computed (at least in principle).
        This computation also provides the map $\Sel_p J \to J(\Q_p)/p J(\Q_p)$.
  \item By post-composing with $\log \otimes \F_p$, we get~$\sigma$, so we
        can check whether $\sigma$ is injective and we can determine
        the image of~$\PP\sigma$.
\end{itemize}
See also Section~\ref{S:Eff} below.


\section{Torsion points on a generic hyperelliptic curve}

Theorem~\ref{T:torsion on generic hyperelliptic curve},
which we hope is of independent interest,
shows that the only torsion points lying on
a generic hyperelliptic curve are its Weierstrass points.

\begin{theorem}
\label{T:torsion on generic hyperelliptic curve}
Let $C$ be a generic hyperelliptic curve of genus $g>1$ 
over a field $k$ of characteristic~$0$; 
i.e., the image of the corresponding morphism
from $\Spec k$ to the moduli space over $\Q$ is the generic point.
Assume that $C$ has a $k$-rational Weierstrass point,
which is used to embed $C$ in its Jacobian $J$.
Then $C(\kbar) \intersect J(\kbar)_{\tors}$ 
consists of only the Weierstrass points.
\end{theorem}

\begin{remark}
See \cite{Chiodo-Eisenbud-Farkas-Schreyer2013}*{Theorem~2.3}
for a related result concerning torsion points on the theta divisor
of a generic hyperelliptic curve.
\end{remark}

Before giving the proof in detail, let us explain the strategy.
If $P$ is a torsion point of order $n$ on $C$,
then so are all its Galois conjugates.
Because of ``big monodromy'', there are many such Galois conjugates.
Taking combinations of these yields a principal divisor
associated to a rational function of low degree on $C$.
Such functions are fixed by the hyperelliptic involution,
and this will force $P$ to be a Weierstrass point.

\begin{proof}
We may assume that $C$ is the curve $y^2=\prod_{i=1}^{2g+1}(x-a_i)$ over 
$k \colonequals \Q(a_1,\ldots,a_{2g+1})$, 
where the $a_i$ are indeterminates,
and that $C$ is embedded in $J$ using the Weierstrass point~$\infty$.
By \cite{ACampo1979}*{Theorem~1}, 
the geometric monodromy group contains 
$\ker(\Sp_{2g}(\Z) \to \Sp_{2g}(\Z/2\Z))$.
For $n \ge 1$, let $I_n$ be the image of 
the geometric monodromy group in $\GL_{2g}(\Z/n\Z)$.
Suppose that $n=2^e m$ where $e \in \Z_{\ge 0}$ and $m$ is odd.
Strong approximation \cite{Kneser1965c}*{Satz~2}
for $\Sp_{2g}(\Z)$ shows that $I_n = I_{2^e} \times I_m$,
where $I_{2^e}$ contains diagonal matrices mapping the
first standard basis element of $(\Z/2^e\Z)^{2g}$
to any odd multiple, and $I_m$ contains $\Sp_{2g}(\Z/m\Z)$,
which acts transitively on points of exact order $m$.

Now suppose that $P \in C(\kbar)$ is a torsion point of exact order $n$.
Then $P$ is the first vector in a $\Z/n\Z$-basis of $J[n]$,
say $P,Q,\ldots$.
Decompose $P$ as $P_{2^e}+P_m$ where $P_{2^e} \in J[2^e]$ and $P_m \in J[m]$.
Decompose $Q$ similarly as $Q_{2^e} + Q_m$.
Then we may find elements of $I_n$ mapping $P = P_{2^e} + P_m$
to $R_1 \colonequals P_{2^e} + Q_m$,
$R_2 \colonequals P_{2^e} + 2P_m$,
and $R_3 \colonequals P_{2^e} + (P_m+Q_m)$.
Then $P + R_3 - R_1 - R_2$, viewed as a divisor on $C_{\kbar}$, is principal.
If $m \ge 3$, then $P,R_1,R_2,R_3$ are all distinct,
so the corresponding rational function is of degree~$2$,
which implies that its divisor is fixed by the hyperelliptic involution,
a contradiction.
Thus $m=1$, so $n=2^e$.
If $e \ge 3$, then we may find elements of $I_n$ mapping $P$
to $S_1 \colonequals (2^{e-2}+1)P$,
$S_2 \colonequals (2 \cdot 2^{e-2}+1)P$,
and $S_3 \colonequals (3 \cdot 2^{e-2}+1)P$,
and then the divisor $P+S_1-S_2-S_3$ yields a contradiction as before.
Thus $e \le 2$, so $n \le 4$.

Now $nP-n\infty$ is the divisor of some $h \in \kbar[x] + \kbar[x] y$.
The valuation $v$ at $\infty$
satisfies $v(x)=-2$ and $v(y)=-(2g+1) < -4$,
so $h \in \kbar[x]$.
Thus the hyperelliptic involution fixes $h$,
so it fixes $nP-n\infty$,
so it fixes $P$.
In other words, $P$ is a Weierstrass point.
\end{proof}

\begin{remark}
Our application to $2$-adic Chabauty 
needs to consider only torsion points $P$ of odd order~$n$
in Theorem~\ref{T:torsion on generic hyperelliptic curve}.
There is a simpler proof in this special case:
the point $Q \colonequals 2P$ is a Galois conjugate of $P$,
and the divisor $Q + \infty - 2P$ is principal.
\end{remark}


\section{Families of curves and Jacobians}
\label{S:families}

\subsection{Abelian scheme over a \texorpdfstring{$p$}{p}-adic variety}

Let $M$ be a smooth variety over $\Q_p$.
Let $\JJ \to M$ be an abelian scheme.
Given $m \in M(\Q_p)$, let $\JJ_m$ be the fiber,
an abelian variety over $\Q_p$.
If $\Omega^1_{\JJ/M}$ is a free $\OO_{\!\JJ}$-module
with basis $\omega_1,\ldots,\omega_g$,
then fiberwise integration of these $1$-forms defines 
a local diffeomorphism $\Log \colon \JJ(\Q_p) \to \Q_p^g \times M(\Q_p)$
that is fiberwise the homomorphism $\log$ of Section~\ref{S:logarithm}.

\begin{definition}
For an open subset $U \subseteq M(\Q_p)$,
let $\JJ(\Q_p)_U$ be its inverse image in $\JJ(\Q_p)$.
By a \defi{trivialization} of $\JJ(\Q_p) \to M(\Q_p)$ above $U$,
we mean a diffeomorphism
\[
	\tau \colon \JJ(\Q_p)_U \to (\Z_p^g \times F) \times U
\]
over $U$, for some finite abelian group $F$, such that 
\begin{itemize}
\item for each $m \in U$, 
the fiber $\tau_m \colon \JJ_m(\Q_p) \to \Z_p^g \times F$ 
is an isomorphism of $p$-adic Lie groups, and
\item after replacing $M$ by a Zariski open subset
whose set of $\Q_p$-points still contains~$U$,
the composition of $\tau$ with the projection to $\Z_p^g \times U$
agrees with a map $\Log$ defined as above.
\end{itemize}
\end{definition}

\begin{proposition}
\label{P:trivialized}
The base $M(\Q_p)$ can be covered by open subsets $U$ 
above which $\JJ(\Q_p) \to M(\Q_p)$ can be trivialized.
\end{proposition}

\begin{proof}
Since $\JJ \to M$ is an abelian scheme,
$\Omega^1_{\JJ/M}$ is trivialized on 
the preimages of some Zariski open sets covering $M$.
Thus we may reduce to the case that $\Omega^1_{\JJ/M}$ is free.
Choose a basis to define $\Log$.

Let $m_0 \in M(\Q_p)$.
Change the basis above so that $\Log(\JJ_{m_0}(\Q_p)) = \Z_p^g$.
Since $\JJ(\Q_p) \to M(\Q_p)$ is a proper and open map,
for each compact open subgroup $H \le \Q_p^g$,
the locus of $m \in M(\Q_p)$
such that $\Log(\JJ_m(\Q_p)) \subseteq H$
is open and closed.
The same is true for the locus where $\Log(\JJ_m(\Q_p))$ \emph{equals} $H$,
because a subgroup of $\Q_p^g$ equals $H$
if and only if it is contained in $H$
and not contained in any of the finitely many maximal subgroups of $H$.
In particular, we can find an open neighborhood $U$ of $m_0$ in $M(\Q_p)$
such that $\Log(\JJ(\Q_p)_U) = \Z_p^g \times U$.
Since $\Log$ is a local diffeomorphism, 
after shrinking $U$, 
we can find analytic sections $s_1,\ldots,s_g$ of $\JJ(\Q_p)_U \to U$
such that $\Log\bigl(s_i(U)\bigr) = \{e_i\} \times U$ for each $i$, 
with $e_i$ the standard basis vector.
Each group $\JJ_m(\Q_p)$ factors canonically as 
a finitely generated $\Z_p$-module
and a prime-to-$p$ finite group;
if we replace each $s_i$ by its projection onto the $\Z_p$-module, fiberwise,
then we may define a fiberwise $\Z_p$-module homomorphism
$\phi \colon \Z_p^g \times U \to \JJ(\Q_p)_U$
sending $e_i$ to $s_i$.

The map from the torsion locus
$\JJ(\Q_p)_{\tors} = \Log^{-1}(\{0\} \times M(\Q_p))$ to $M(\Q_p)$
is proper (since $\JJ(\Q_p) \to M(\Q_p)$ is proper)
and a local diffeomorphism (since $\Log$ is),
so it is a locally constant family of finite abelian groups over $M(\Q_p)$.
Shrink $U$ so that this family is constant, say equal to $F$, above $U$.
Thus we obtain a fiberwise homomorphism 
$\psi \colon F \times U \to \JJ(\Q_p)_U$.
The product of $\phi$ and $\psi$ (over $U$)
is a fiberwise isomorphism $(\Z_p^g \times F) \times U \to \JJ(\Q_p)_U$
whose inverse is a trivialization $\tau$ above $U$.
\end{proof}

\subsection{The universal family of hyperelliptic curves}

Recall from Section~\ref{S:notation}
the notation $\calF_g(R)$ for the set of all nice hyperelliptic
curves of odd degree and genus~$g$ 
in standard form with coefficients in~$R$.
Let $\calM$ be the moduli space over $\Q$ 
such that $\calM(k)=\calF_g(k)$ for each field extension $k/\Q$;
more precisely, 
$\calM$ is the complement of the discriminant locus $\Delta=0$ in $\Aff^{2g+1}$.
Then $\calF_g(\Z_p) = \Z_p^{2g+1} \intersect \calM(\Q_p)$.
Endow $\calF_g(\Z_p)$ with the Haar measure from $\Z_p^{2g+1}$.
By a \defi{congruence class} in $\calF_g(\Z_p)$,
we mean a coset $U$ of $(p^e \Z_p)^{2g+1}$ in $\Z_p^{2g+1}$
with $U \subseteq \calF_g(\Z_p)$;
for such $U$, the density of $\calF_g \intersect U$ equals
the measure of $U$.

Let $\pi \colon \CC \to \calM$ be the universal curve;
$\pi$ is a smooth proper morphism whose fibers
are nice hyperelliptic curves of genus~$g$.
Its relative Jacobian is an abelian scheme $\JJ \to \calM$ 
\cite{Bosch-Lutkebohmert-Raynaud1990}*{p.~260, Proposition~4}.
Call a congruence class $U \subseteq \calF_g(\Z_p)$ \defi{trivializing}
if $\JJ(\Q_p) \to \calM(\Q_p)$ is trivialized above $U$.

\begin{lemma}
\label{L:disjoint union}
Any open subset of $\calF_g(\Z_p)$
is a disjoint union of trivializing congruence classes.
\end{lemma}

\begin{proof}
The congruence classes in $\calF_g(\Z_p)$ form a basis for its topology.
By Proposition~\ref{P:trivialized},
the same is true for the trivializing congruence classes.
Thus any open subset $V$ of $\calF_g(\Z_p)$ is a union of
trivializing congruence classes $U_i$.
If two congruence classes meet, then one contains the other.
Thus $V$ is the disjoint union of the $U_i$ not contained in a larger one. 
\end{proof}

\begin{proposition}
\label{P:torsion on curves over Q}
The density of $C \in \calF_g$ such that $J(\Q)_{\tors} \ne 0$ is zero.
\end{proposition}

\begin{proof}
By Lemma~\ref{L:disjoint union}, 
there is a finite disjoint union $\calU$ of trivializing congruence classes $U$
such that the measure of $\calU$,
or equivalently the density of $\calF_g \intersect \calU$,
is as close as desired to~$1$.
Thus it suffices to prove the result for the $C \in \calF_g$ 
belonging to one trivializing congruence class~$U$.
For such $C$, the size of $J(\Q_p)_{\tors}$ is constant, say $n$.
The monodromy action does not fix any nonzero torsion point
on the geometric generic fiber of $\JJ \to \calM$,
so the Hilbert irreducibility theorem 
shows that the density of such $C$ such that $J(\Q)$ has a nonzero
point of order dividing $n$ is zero.
\end{proof}

Using the section $\infty \colon \calM \to \CC$,
we identify $\CC$ with a closed subscheme of $\JJ$.
Let $\WW \injects \CC$ be the locus of Weierstrass points.
Then $\WW$ is Zariski open and closed in $\JJ[2]$.
Let $Z$ be the set of $m \in \calF_g(\Z_p)$
such that $\CC_m(\Q_p) \intersect \JJ_m(\Q_p)_{\tors}$
is \emph{not} contained in $\WW_m(\Q_p)$.

\begin{proposition}
\label{P:odd torsion on curves}
The set $Z$ is closed in $\calF_g(\Z_p)$ and is of measure zero.
\end{proposition}

\begin{proof}
By Lemma~\ref{L:disjoint union} applied to $\calF_g(\Z_p)$,
we may restrict attention to a trivializing congruence class $U$.
Let $n$ be the constant value of $\#\JJ_m(\Q_p)_{\tors}$ for $m \in U$.
Let $Z'$ be the image of the $\Q_p$-points
under the restriction
$\pi_n \colon \CC \intersect (\JJ[n]-\WW) \to \calM$ of $\pi$.
Then $Z \intersect U = Z' \intersect U$.
Since $\CC \to \calM$ is proper
and $\JJ[n] \setminus \WW \to \calM$ is finite \'etale,
the morphism $\pi_n$ is proper,
so $Z'$ is closed.
By Theorem~\ref{T:torsion on generic hyperelliptic curve},
$\pi_n$ is not dominant, 
so $Z'$ is of measure zero.
\end{proof}

\begin{corollary}
\label{C:odd torsion on curves}
The set of $C \in \calF_g$ 
such that $C(\Q_p) \intersect J(\Q_p)_{\tors}$ contains a non-Weierstrass point
is of density zero.
\end{corollary}

\begin{proof}
Apply Lemma~\ref{L:disjoint union} to $\calF_g(\Z_p) \setminus Z$
to show that the complement has density $1$.
\end{proof}

\begin{proposition}
\label{P:locally constant image of rholog}
Let $U$ be a trivializing congruence class.
Let $Z$ be as in Proposition~\ref{P:odd torsion on curves}.
Let $U' \colonequals U \setminus Z$.
Then $\rholog(\CC_m(\Q_p))$ in $\PP^{g-1}(\F_p)$
is locally constant as $m$ varies in $U'$.
\end{proposition}

\begin{proof}
We have analytic maps of $p$-adic manifolds
\begin{equation}
  \label{E:maps of manifolds}  
	\CC(\Q_p)_{U'} \To 
	\JJ(\Q_p)_{U'} \stackrel{\Log}\To
	\Z_p^g \times {U'} \To
	\Z_p^g \stackrel{\PP}\dashrightarrow 
	\PP^{g-1}(\Q_p) \stackrel{\rho}\To
	\PP^{g-1}(\F_p),
\end{equation}
except that the dashed arrow is indeterminate at $0$.
We resolve the indeterminacy by blowing up the first four manifolds
along the inverse image of $0 \in \Z_p^g$.
The inverse image in $\CC(\Q_p)_{U'}$ is the torsion locus,
which by definition of $U'$ equals only $\WW(\Q_p)_{U'}$,
which is a smooth \emph{divisor} on $\CC(\Q_p)_{U'}$.
Thus the blow-up of $\CC(\Q_p)_{U'}$ is $\CC(\Q_p)_{U'}$
itself, so \eqref{E:maps of manifolds} extends to a continuous map
$e \colon \CC(\Q_p)_{U'} \to \PP^{g-1}(\F_p)$.
The fibers of $e$ are open and closed.
So are their images in $U'$ 
since $\CC \to \calM$ is smooth and proper.
The locus in $U'$ where $e(\CC_m(\Q_p))$ equals a 
given subset of $\PP^{g-1}(\F_p)$
is a finite Boolean combination of these images,
hence again open and closed.
Thus $e(\CC_m(\Q_p))$ is locally constant as $m$ varies in $U'$.
Finally, $\rholog$ is just the restriction of $e$ to 
a dense open subset of $\CC_m(\Q_p)$, and $e$ is locally constant,
so $\rholog(\CC_m(\Q_p)) = e(\CC_m(\Q_p))$.
\end{proof}

\begin{remark}
The set $\rholog(\CC_m(\Q_p))$ is generally \emph{not} 
locally constant in a neighborhood of points of $Z$.
\end{remark}

\subsection{Equidistribution of Selmer elements}
\label{S:equidistribution}

Let $U$ be a trivializing congruence class.
If $m \in \calF_g \intersect U$, 
the trivialization 
lets us construct the diagram~\eqref{E:big} for 
$C \colonequals \CC_m$ and $J \colonequals \JJ_m$;
in particular, we obtain $\sigma \colon \Sel_p J \to \F_p^g$.

For each $g \ge 1$ and prime $p$, 
we may now formulate the following equidistribution conjecture, 
compatible with both the heuristics in \cite{Poonen-Rains2012-selmer}
and the theorems for $p=2$ in \cite{Bhargava-Gross-preprint}.

\begin{conjectureEq}
For any trivializing congruence class $U$
and any $w \in \F_p^g$,
the average size of 
$\{s \in \Sel_p J \setminus \{0\} : \sigma(s)=w\}$
as $C$ varies in $\calF_g \intersect U$
is $p^{1-g}$.
\end{conjectureEq}

\begin{theorem}
\label{T:Eq(g,2)}
For each $g \ge 1$, $\Eq_g(2)$ holds.
\end{theorem}

\begin{proof}
By \cite{Bhargava-Gross-preprint}*{Theorem~11.1}
(adapted as in Remark~\ref{R:a_1=0} below), 
the average size of $\Sel_2 J \setminus \{0\}$ 
for $C \in \calF_g \intersect U$ is $2$.
The trivialization identifies each group $J(\Q_2)$
with a fixed group $\Z_2^g \times F$,
and \cite{Bhargava-Gross-preprint}*{Theorem~12.4}
states that the images of the nonzero Selmer elements 
under $\Sel_2 J \to J(\Q_2)/2J(\Q_2) \isom \F_2^g \times F/2F$
are equidistributed in $\F_2^g \times F/2F$, 
so their images under the projection to $\F_2^g$ are equidistributed too.
Thus on average there are $2/\#\F_2^g = 2^{1-g}$ nonzero Selmer elements 
mapping to a given element of $\F_2^g$.
\end{proof}

\begin{remark}
\label{R:a_1=0}
Let $\calF'_g$ be the subset of $\calF_g$ where $a_1=0$,
and let $\calF''_g$ be the subset of $\calF'_g$
where there is no prime $p$ with $p^{2m} \mid a_m$ for all $m$. 
The paper \cite{Bhargava-Gross-preprint} works not with $\calF_g$,
but with $\calF''_g$.
Here we explain how to transfer the equidistribution results
from $\calF''_g$ to $\calF_g$.

First, given subsets
\[
	B_\infty \colonequals 
	[b_2,b_2'] \times \cdots \times [b_{2g+1},b_{2g+1}'] 
	\subseteq \R^{2g},
\]
and $B_p$ a coset of a finite-index subgroup of $\Z_p^{2g}$ 
for $p$ in some finite set $S$,
call $B \colonequals B_\infty \times \prod_{p \in S} B_p$ a \defi{box}.
For $X>0$, consider the curves in $\calF'_g$ 
whose coefficient tuple satisfies
$[a_2/X^2,\ldots,a_{2g+1}/X^{2g+1}] \in B_\infty$
and $(a_2,\ldots,a_{2g}) \in B_p$ for all $p \in S$.
The arguments of \cite{Bhargava-Gross-preprint} 
carry over essentially without change
to prove equidistribution of Selmer elements
for such curves as $X \to \infty$.
By taking finite linear combinations,
one obtains a variant that counts the same curves
but with a weight that is a step function that is 
finitely piecewise constant on sub-boxes of $B$.
Now if $w$ is any bounded weight function on~$B$ such for
every $\epsilon > 0$, there is a step function $s$
as above such that $|w - s| \le \epsilon$ outside a finite union
of sub-boxes that has total measure less than~$\epsilon$,
then by comparing with~$s$ and taking the limit
as $\epsilon \to 0$, we can count both curves and Selmer elements
weighted according to~$w$.
In particular, for nonnegative $w$ bounded below by a positive constant
on some sub-box, we can deduce an analogous Selmer equidistribution result.

Let $S$ be a finite set of primes including those dividing $2g+1$.
The ``complete the $(2g+1)^{\textup{st}}$ power and scale''
map sending $f(x)$ to $(2g+1)^{2g+1} f((x-a_1)/(2g+1))$
defines a map on coefficient tuples
from $[-1,1]^{2g+1} \times \prod_{p \in S} \Z_p^{2g+1}$
into some box $B$,
and the pushforward of the uniform measure is a weight function $w$ as above.
The corresponding map from $\calF_{g,X}$ to $\calF'_g$
has fibers whose size is approximately proportional to $w$
at the corresponding $f \in \calF'_g$, normalized,
so the limit of the average for $\calF_g$
equals the limit of the weighted average for~$\calF'_g$.
(One can also impose finitely many congruence conditions
by replacing each $\Z_p^{2g+1}$ by a coset of a finite-index subgroup.)
\end{remark}

\begin{remark}
Although $\Eq_g(p)$ is about $p$-adic equidistribution of $p$-Selmer elements,
one could also ask about $v$-adic equidistribution of $p$-Selmer elements
for any place $v$ of $\Q$.
In fact, \cite{Bhargava-Gross-preprint}*{Theorem~12.4} proves
$v$-adic equidistribution of $2$-Selmer elements for all $v$,
and can even handle finitely many $v$ simultaneously.
\end{remark}

\subsection{A general density result}
\label{S:general density result}

\begin{proposition} 
\label{P:general density result}
Let $U$ be a trivializing congruence class
such that the subset $\rholog(\CC_m(\Q_p))$ of~$\PP^{g-1}(\F_p)$ 
is constant for $m \in U$, say equal to $I$.
Then, for $C \in \calF_g \intersect U$ 
outside a subset of relative upper density at most $(1+\#I)p^{1-g}$,
we have:
\begin{enumerate}[\upshape (i)]
\item if $p=2$, then $C(\Q) = C(\Q_2) \intersect \overline{J(\Q)} = \{\infty\}$.
\item if $p>2$ and $\Eq_g(p)$ holds, then
$C(\Q)=\{\infty\}$ and $C(\Q_p) \intersect \overline{J(\Q)}$
consists of Weierstrass points.
\end{enumerate}
\end{proposition}

\begin{proof}
By Theorem~\ref{T:Eq(g,2)}, $\Eq_g(2)$ holds;
if $p>2$, assume $\Eq_g(p)$.
Then for $C \in \calF_g \intersect U$, 
the average number of nonzero elements of $\Sel_p J$ 
mapped by $\sigma$ to $0$ is $p^{1-g}$,
so the relative upper density of $C$ having such a Selmer element
is at most $p^{1-g}$.
Similarly, the average number of nonzero elements of $\Sel_p J$
mapped by $\PP\sigma$ into $I$ is $(p-1)\#I  p^{1-g}$,
but each curve with such an element has at least $p-1$ such elements
(its nonzero multiples),
so the relative upper density of such curves is 
at most $\#I p^{1-g}$.
Together, these have relative upper density at most $(1+\#I)p^{1-g}$,
and this is unchanged if we include the
density zero sets of Proposition~\ref{P:torsion on curves over Q}
and Corollary~\ref{C:odd torsion on curves}.

For the \emph{other} $C \in \calF_g \intersect U$,
Proposition~\ref{P:rational points are odd}
states that 
$C(\Q_p) \intersect \overline{J(\Q)} \subseteq J(\Q_p)[p']$.
Since we excluded the set of Corollary~\ref{C:odd torsion on curves},
$C(\Q_p) \intersect \overline{J(\Q)}$ consists of Weierstrass points.
For $p=2$, these together imply
$C(\Q_p) \intersect \overline{J(\Q)} = \{\infty\}$.
For all $p$, 
our exclusion of the set of Proposition~\ref{P:torsion on curves over Q} 
implies $C(\Q)=\{\infty\}$.
\end{proof}

\begin{remark}
\label{R:total measure 1}
By Lemma~\ref{L:disjoint union}
and Proposition~\ref{P:locally constant image of rholog},
there is a disjoint union of sets $U$ satisfying the hypothesis of 
Proposition~\ref{P:general density result}
and having total measure~$1$.
\end{remark}

\begin{remark}
We expect that for $p>2$, there is a positive density of 
$C \in \calF_g$ such that $C(\Q_p) \intersect \overline{J(\Q)}$
is strictly larger than $\{\infty\}$.
The reason is that we expect that there is a positive density of $C$
such that $C$ has good reduction at $p$,
there is a Weierstrass point $W \in C(\Q_p)\setminus C(\Q)$,
and there is a point $P \in J(\Q)$ with the same reduction as $W$.
For such curves, $p^n P \to W$ as $n \to \infty$,
so $W \in C(\Q_p) \intersect \overline{J(\Q)}$.
\end{remark}


\section{Average number of residue disks}

In this section, we fix $g \ge 1$ and a prime $p$.
For a random $C \in \calF_g(\Z_p)$,
let $\calC$ be its minimal proper regular model.
The main goal of this section is the following.

\begin{theorem} \label{T:average number of smooth points}
We have $\E \#\calC^{\smooth}(\F_p) \le p+1$.
\end{theorem}

Because it is difficult to construct $\calC$ explicitly,
we construct a model with a weaker property.
Call a $\Z_p$-model $\calD$ of $C$ \defi{decent}
if $\calD$ is proper over $\Z_p$
and the image of $C(\Q_p) = \calD(\Z_p) \to \calD(\F_p)$
is contained in $\calD^{\smooth}(\F_p)$.

\begin{lemma}
\label{L:non-minimal model}
If $\calD$ is decent, 
then $\#\calC^{\smooth}(\F_p) \le \#\calD^{\smooth}(\F_p)$.
\end{lemma}

\begin{proof}
Let $\pi \colon \calE \to \calD$ be the minimal desingularization of $\calD$;
this is an isomorphism above~$\calD^{\smooth}$.
Also, if $e \in \calE^{\smooth}(\F_p)$,
then $e$ is the reduction
of a point in $C(\Q_p)$ by Hensel's lemma,
and $\calD$ is decent, 
so $\pi(e) \in \calD^{\smooth}(\F_p)$.
Thus $\pi$ defines a bijection $\calE^{\smooth}(\F_p) \to \calD^{\smooth}(\F_p)$.
On the other hand, $\calE \to \calC$ factors as a sequence of blow-ups
at closed points \cite{Lichtenbaum1968}*{II.A, Theorem~1.15}, 
and each blow-up morphism is surjective on $\F_p$-points.
\end{proof}

By Lemma~\ref{L:non-minimal model},
to prove Theorem~\ref{T:average number of smooth points} 
it suffices to construct a decent model $\calD$ of each $C \in \calF_g(\Z_p)$
and to prove $\E \#\calD^{\smooth}(\F_p) \le p+1$.
We use the following recursive algorithm to construct $\calD$.

\bigskip

\Algorithm $\MakeDecentModel(C)$.

Input: A curve $C \colon y^2=f(x)$ in $\calF_g(\Z_p)$.
\begin{enumerate}[\upshape 1.]
\item
Let $\calD$ be the standard compactification of $y^2=f(x)$ over $\Z_p$.
\item
Let $U$ be the closed subscheme $y^2=f(x)$ of $\Aff^2_{\Z_p}$
with its open immersion into $\calD$.
\item
Modify $\calD$ by running $\Fix(U)$ below.
\end{enumerate}

\bigskip

\Subroutine $\Fix(U)$.

Input: A closed subscheme $U \colon y^2=h(x)$ of $\Aff^2_{\Z_p}$
with an open immersion into $\calD$.
\begin{enumerate}[\upshape 1.]
\item Replace $\calD$ by its blow-up at the set of non-regular $\F_p$-points
of $U$.
\item For each $c \in \{0,1,\ldots,p-1\}$, 
if $p \mid h'(c)$ and $p^2 \mid h(c)$, 
then let $U_c$ be $y^2=p^{-2} h(c+px)$,
which is an affine patch of the blown-up $\calD$,
and run $\Fix(U_c)$.
(These processes for different $c$ 
may be run independently without interference,
since the special fibers of the $U_c$ have disjoint images in $\calD$.)
\end{enumerate}

\begin{lemma}
\label{L:terminates decently}
Algorithm $\MakeDecentModel(C)$ terminates
and yields a decent model $\calD$ of~$C$.
\end{lemma}

\begin{proof}
If the recursion reaches nesting depth $n$
(where the initial call to $\Fix(U)$ is nesting depth $0$),
then the composition of the changes of variable $x \mapsto c+px$
is of the form $x \mapsto d + p^n x$ for some $d \in \{0,1,\ldots,p^n-1\}$
such that $p^{-2n} f(d + p^n x) \in \Z_p[x]$,
and hence $f(d) \in p^{2n} \Z_p$ and $f'(d) \in p^n \Z_p$.
Thus if the nesting is unbounded,
compactness yields $d \in \Z_p$ such that $f(d)=f'(d)=0$,
contradicting the definition of $\calF_g(\Z_p)$.
Hence the algorithm terminates.

Suppose that $P \in C(\Q_p)$.
If $x(P) \notin \Z_p$, then $P$ reduces to the smooth point $\infty$
on the special fiber of $\calD$.
Otherwise, $P$ belongs to $U(\Z_p)$ for the initial $U$.
Consider the \emph{last} time $\Fix(U)$ is called with a $U$
such that $P \in U(\Z_p)$.
Without loss of generality, make a change of variables $x \mapsto x+c$
to assume that $P$ reduces to a point in $U(\F_p)$ with $x=0$.
Let $a=h(0)$ and $b=h'(0)$.
\begin{enumerate}
\item Suppose $p\nmid b$.
Then $U$ is smooth at the $\F_p$-points with $x=0$
(the $x$-derivative is nonzero).
\item Suppose $p \mid b$. 
  \begin{enumerate}
  \item Suppose $p\nmid a$.
    \begin{enumerate}
    \item Suppose $p\ne 2$.
Then $U$ is smooth at the $\F_p$-points with $x=0$
(the $y$-derivative is nonzero).
    \item Suppose $p=2$.
	Then $U$ is isomorphic to $y^2+2y = (a-1) + bx + \cdots$,
and has a unique $\F_2$-point $u$ with $x=0$.
      \begin{enumerate}
      \item If $a \equiv 3 \pmod{4}$, 
		then $U$ is regular but not smooth at $u$, 
		so $P$ could not have existed.
      \item If $a \equiv 1 \pmod{4}$, then $U$ is not regular at $u$,
so $u$ was blown up in Step~1 of $\Fix(U)$;
then $P$ corresponds to a $\Z_2$-point of the affine patch
$U' \colon y^2+y=a' + b' x + \cdots$ of the blow-up
obtained by making the change of variable $(x,y) \mapsto (2x,2y)$
and dividing by $2^2$;
this entire patch is smooth (the $y$-derivative is nonvanishing).
      \end{enumerate}
    \end{enumerate}
  \item Suppose $p \mid a$ but $p^2 \nmid a$. 
Then $U$ is regular but not smooth at the unique $\F_p$-point with $x=0$
(the origin), so $P$ could not have existed.
  \item Suppose $p^2 \mid a$. 
Then in Step~2 of $\Fix(U)$, we would have called $\Fix(U_0)$
for $U_0$ such that $P \in U_0(\Z_p)$,
contradicting the assumption on $U$.\qedhere
  \end{enumerate}
\end{enumerate}
\end{proof}

Construct $\calD$ by algorithm $\MakeDecentModel$.
For $n \ge 0$, let $\overline{\calH}_n$ be the set 
of polynomials 
\[
	f(x) \colonequals 
	p^{(2g-1)n} x^{2g+1} + p^{(2g-2)n} a_{2g} x^{2g} + \cdots 
	+ p^n a_3 x^3 + a_2 x^2 + a_1 x + a_0
\]
with $a_0,\ldots,a_{2g} \in \Z_p$,
and let $\calH_n$ be the (full measure) subset with nonzero discriminant.
Identify each $f \in \calH_n$ 
with the standard compactification of $y^2=f(x)$ over $\Q_p$;
for example, $\calH_0 \isom \calF_g(\Z_p)$.
Let 
$\overline{\calS}_n \colonequals
   \{ h \in \overline{\calH}_n : p \mid a_1 \textup{ and } p^2 \mid a_0 \}$;
define $\calS_n$ similarly.
The bijection $\overline{\calS}_n \to \overline{\calH}_{n+1}$
sending $h(x)$ to $p^{-2} h(px)$
respects addition, so it respects Haar measure (up to normalization).
Thus, by induction on $n$,
inside a call to subroutine $\Fix(U)$ at nesting depth $n$
arising from a sequence of choices $c_1,\ldots,c_n$
in Step~2 of earlier calls,
the distribution of $h$ is uniform over $\calH_n$.
Let $X_n$ be the random variable on $\calH_n$ that counts the
number of smooth $\F_p$-points of the final $\calD$
lying above $\F_p$-points in this $U$ having $x=0$;
if we replaced $0$ by any other $c \in \{0,1,\ldots,p-1\}$,
the distribution of values would be the same.

\begin{lemma}
\label{L:random variable X_n}
For $n \ge 0$,
the restriction $X_n|_{\calS_n}$ is the sum of $p$ random variables,
each of which has the same distribution of values on $\calS_n$
as $X_{n+1}$ has on $\calH_{n+1}$.
\end{lemma}

\begin{proof}
For $h \in \calS_n$, 
Step~2 of $\Fix(U)$ leads to $U' \colon y^2=p^{-2} h(px)$,
and the points of $\calD^{\smooth}(\F_p)$ lying above points in $U$
with $x=0$ are all those lying above $U'$.
The number of these whose image in $U'$ has a particular
$x$-coordinate in $\F_p$ is distributed like $X_{n+1}$ on $\calH_{n+1}$.
\end{proof}

\begin{lemma}
\label{L:E X_n}
For $n \ge 0$, we have $\E X_n = 1$.
\end{lemma}

\begin{proof}
We divide $\calH_n$ into subsets corresponding to the cases
in the proof of Lemma~\ref{L:terminates decently}.
For each case, we compute its probability,
and the average contribution to $X_n$
conditioned on being in that case:
\begin{center}
  \begin{tabular}{c|cc}
Case & Probability & Average contribution to $X_n$ \\ \hline
(1){\large\strut} & $1-p^{-1}$ & $1$ \\ 
(2)(a){\strut} & $p^{-1}-p^{-2}$ & $1$ \\
(2)(b){\strut} & $p^{-2}-p^{-3}$ & $0$ \\
(2)(c){\strut} & $p^{-3}$ & sum of $p$ copies of $X_{n+1}$. \\
  \end{tabular}
\end{center}
Let us explain the entries in the last column.
In case~(1), the smooth $\F_p$-points of $U$ with $x=0$
correspond to square roots
of a uniformly random element of $\F_p$;
the expected number of square roots is~$1$.
In case~(2)(a) for $p\ne 2$, the contribution is
the expected number of square roots of a random element of $\F_p^\times$,
which is again~$1$.
In case~(2)(a) for $p=2$, in subcase~(A) the contribution is $0$
(we are counting smooth $\F_p$-points),
while in subcase~(B) the average contribution is~$2$
since each point of $\Aff^2(\F_p)$ has a $1/2$ chance of
lying in $U'(\F_p)$;
thus the overall average contribution in case~(2)(a) for $p=2$ is again~$1$.
In case~(2)(b), the contribution is~$0$.
Case~(2)(c) corresponds to $\calS_n$,
so we use Lemma~\ref{L:random variable X_n}.
Moreover, the analysis shows that $X_n \le 4$ outside of case~(2)(c).

The upshot is that $X_n$ is given by a process whose parameters
are independent of $n$;
it would be a Bienaym\'e--Galton--Watson process if the 
$p$~random variables in Lemma~\ref{L:random variable X_n} were independent 
(they are generally not).
For real $B>4$, the only way that $X_n \ge B$ can hold
is if we are in case~(2)(c) and one of the $p$ copies of $X_{n+1}$
exceeds $B/p$;
thus $\Prob(X_n \ge B) \le p^{-3} p \Prob(X_{n+1} \ge B/p)$.
Iterate this $k$ times, where $k$ is the first integer with $B/p^k \le 4$,
to obtain
\[
	\Prob(X_n \ge B) \;\le\; 
	p^{-2k} \Prob(X_{n+k} \ge B/p^k) \;\le\; 
	p^{-2k} \;=\; 
	O(B^{-2}).
\]
Thus the average $\E X_n = \sum_{B=1}^\infty \Prob(X_n \ge B)$
is finite and bounded independently of $n$.
The table implies that
\[
	\E X_n = (1-p^{-1}) \cdot 1 + (p^{-1}-p^{-2}) \cdot 1 + (p^{-2}-p^{-3}) \cdot 0 + p^{-3} (p \cdot \E X_{n+1}).
\]
By induction on $k$, 
we obtain $\E X_n = 1 - p^{-2k} + p^{-2k} \cdot \E X_{n+k}$.
Taking the limit as $k \to \infty$ yields $\E X_n = 1$.
\end{proof}

\begin{lemma}
\label{L:average of D^smooth(F_p)}
We have $\E \#\calD^{\smooth}(\F_p) = p+1$.
\end{lemma}

\begin{proof}
First, $\calD$ has the smooth $\F_p$-point $\infty$ on
the standard compactification.
The other points of $\calD^{\smooth}(\F_p)$ lie above the initial $U$.
The average number of these lying above $\F_p$-points in $U$ with $x=c$ 
is independent of $c \in \F_p$,
so it equals its value for $c=0$,
which by definition is $\E X_0$,
which is $1$ by Lemma~\ref{L:E X_n}.
Thus $\E \#\calD^{\smooth}(\F_p) = 1 + \sum_{c \in \F_p} 1 = p+1$.
\end{proof}

This completes the proof of Theorem~\ref{T:average number of smooth points}.

\begin{remark}
One can show that there is a positive probability that 
the morphism $\calE \to \calC$ 
in the proof of Lemma~\ref{L:non-minimal model} involves
blowing up a smooth $\F_p$-point;
thus $\E \#\calC^{\smooth}(\F_p) < p+1$.
\end{remark}

\begin{remark}
On the other hand, at least for $p>2$,
one can show that a random $C \in \calF_g$ \emph{with good reduction} 
has $\E \#\calC^{\smooth}(\F_p) = p+1$.
(This is because all such $C$ arise from a curve in $\calF_g$
with discriminant in $\Z_p^\times$ by a substitution $x \mapsto d+p^n x$
for a uniquely determined $n \ge 0$ and $d \in \{0,1,\ldots,p^n-1\}$,
and the expected value of $C(\F_p)$ for $C \in \calF_g(\F_p)$ is $p+1$,
as one sees by grouping each $C$ with its quadratic twist.)
Thus one has the counterintuitive fact that on average,
curves with bad reduction have fewer smooth $\F_p$-points
than curves with good reduction!
\end{remark}

\begin{remark}
The argument proving Theorem~\ref{T:average number of smooth points}
proves $\E \#\calC^{\smooth}(\F_p) \le p+1$
also for the random nice curve over $\Q_p$ given by
\[
	y^2 = a_{2g+2} x^{2g+2} + \cdots + a_0
\]
for $a_0,\ldots,a_{2g+2} \in \Z_p$
such that the discriminant is nonzero.
\end{remark}

\begin{corollary}
\label{C:average image of rholog}
For $C \in \calF_g(\Z_p)$,
the average size of $\rholog(C(\Q_p))$ is at most 
\[
	\begin{cases}
		6g + 9, &\textup{if $p=2$,} \\
		\dfrac{p^2-p}{p-2}(2g-2) + (p+1)^2, &\textup{if $p>2$.} \\
	\end{cases}
\]
\end{corollary}

\begin{proof}
Combine Proposition~\ref{P:image of general C}
and Theorem~\ref{T:average number of smooth points}.
\end{proof}

\begin{remark} \label{R:better rholog bound}
  Working with residue disks defined in terms of decent models and making use
  of the fact that $\rholog \circ \iota = \rholog$, where $\iota$ is the
  hyperelliptic involution, one could improve this to $3g + 9/2$ for $p = 2$ and
  \[ \frac{p^2-p}{p-2}(g-1) + \frac{(p+1)^2}{2} \]
  for $p>2$. We omit the details, since they are somewhat technical
  and yield only a modest improvement in our main results.
\end{remark}


\section{The main results}

\subsection{Positive density for \texorpdfstring{$g \ge 3$}{g at least 3}}

\begin{lemma}
\label{L:log for good reduction}
Let $\calJ$ be an abelian scheme over $\Z_p$.
Define $\log \colon \calJ(\Q_p) \to \Q_p^g$ 
by integrating a $\Z_p$-basis of $\HH^0(\calJ,\Omega^1_{\calJ/\Z_p})$.
If $\calJ(\F_p)[p]=0$ and $\calJ(\Q_p)[p]=0$, then $\log(\calJ(\Q_p))=(p\Z_p)^g$.
\end{lemma}

\begin{proof}
Let $F$ be the formal group of $\calJ$ over $\Z_p$.
For $e \ge 1$, 
define 
\[ K_e \colonequals F((p^e\Z_p)^g) = \ker(\calJ(\Z_p) \to \calJ(\Z/p^e\Z)). \]
Since $K_1$ is a pro-$p$-group and $\calJ(\Q_p)[p]=0$,
we have $(K_1)_{\tors}=0$.
Thus $\log|_{K_1}$ is injective.
Hence for $e \ge 2$,
\[
	(\log K_{e-1}:\log K_e) = (K_{e-1}:K_e) = ((p^{e-1}\Z_p)^g : (p^e\Z_p)^g) = p^g. 
\]
Also, $p \log K_{e-1} \subseteq \log K_e$.  
On the other hand, $\log K_e$ is a compact open subgroup of $\Q_p^g$,
and hence a free $\Z_p$-module of rank $g$.
The previous three sentences show that $\log K_{e-1} = p^{-1} \log K_e$.

We prove that $\log K_e = (p^e \Z_p)^g$ for all $e \ge 1$,
by reverse induction on $e$.
For large $e$, we have $\log K_e = (p^e \Z_p)^g$ 
since the derivative of the composition
$(p\Z_p)^g \To \calJ(\Q_p) \stackrel{\log}\To \Q_p^g$
at~$0$ is invertible over $\Z_p$.
The previous paragraph lets us pass from $e$ to $e-1$.

In particular, $\log K_1=(p\Z_p)^g$.
Since $(\calJ(\Q_p):K_1) = \#\calJ(\F_p)$, which is prime to $p$,
we have $\log(\calJ(\Q_p)) = (p\Z_p)^g$ as well.
\end{proof}

\begin{lemma}
\label{L:rholog of size 1}
For each $g>1$, there exists $C \in \calF_g \setminus Z$
such that $\#\rholog(C(\Q_2))=1$.
\end{lemma}

\begin{proof}
Let $C \in \calF_g$ be a curve isomorphic to $y^2 + y = x^{2g+1} + x + 1$.
Completing the square yields a new equation for $C$ of the form $y^2=f(x)$.
The $2$-adic Newton polygon of $f$ is a single line segment from 
$(0,-2)$ to $(2g+1,0)$, which has no interior lattice points,
so $f$ is irreducible over $\Q_2$.
Elements of $J(\overline{\Q}_2)[2]$ correspond to 
partitions of the set of zeros of $f$ into two parts, 
but only the trivial partition is $\Gal(\Qbar_2/\Q_2)$-invariant,
so $J(\Q_2)[2]=0$.

The curve $C$ has good reduction at $2$.
Let $\calC$ and $\calJ$ be the smooth proper models of $C$ and $J$ over $\Z_2$.
The hyperelliptic involution of $\calC_{\F_2}$ is $(x,y) \mapsto (x,y+1)$,
which fixes only $\infty$;
this implies that $\calJ(\Fbar_2)[2]=0$,
so $\calJ(\F_2)[2]=0$.
Lemma~\ref{L:log for good reduction} implies
that a basis of $\HH^0(\calJ,\Omega^1_{\calJ/\Z_2})$
defines $\log$ such that $\log(J(\Q_2)) = (2\Z_2)^g$.
We may divide by $2$ if desired to make $\log(J(\Q_2)) = \Z_2^g$,
but this will not change the map $\rholog$.

The change of variable $s \colonequals 1/x$ and $t \colonequals y/x^{g+1}$
rewrites the equation of $\calC$ as 
\[ t^2 + s^{g+1} t = s^{2g+2} + s^{2g+1} + s, \]
with $\infty$ corresponding to $(0,0)$ in the new model.
Let 
\[ \omega_1 \colonequals \frac{dt}{-(g+1)s^g t + (2g+2)s^{2g+1} + (2g+1)s^{2g} + 1}, \]
the denominator being a partial derivative of the curve equation.
Then $\omega_j \colonequals s^{j-1} \omega_1$ for $j=1,\ldots,g$
form a basis for $\HH^0(\calC,\Omega^1_{\calC/\Z_2})$.
We have $\calC(\F_2)=\{\infty\}$, so there is just one residue disk,
and $t$ is a uniformizer for it.
Expanding in power series in $t$,
we find $s=t^2+t^{2g+3}+\cdots$, 
$\omega_1 = (1 + 0t + \cdots) \, dt$,
and $\omega_j = (t^{2j-2} + \cdots) \, dt$
for $j=1,\ldots,g$,
with all coefficients in $\Z_2$.
Integrating yields
\[
	\boldl(t) = (t + 0t^2 + \cdots, t^3/3 + \cdots, \cdots, t^{2g-1}/(2g+1) + \cdots).
\]
After removing the common factor of $t$, 
for $t \in 2\Z_2$ each summand on the right is in $2\Z_2$
except the initial $1$; thus $\rho(\boldl(t)) = (1:0:\cdots:0)$
for all nonzero $t \in 2\Z_2$.
Hence $\rholog(C(\Q_2))=\{(1:0:\cdots:0)\}$.
The computation shows also that the only zero of $\log$ on $C(\Q_2)$
is $\infty$, so $C(\Q_2) \intersect J(\Q_2)_{\tors} = \{\infty\}$,
so $C \notin Z$.
\end{proof}

\begin{theorem}
\label{T:positive density}
Fix $g \ge 3$.
The lower density of the set of curves $C \in \calF_g$ 
satisfying $C(\Q) = C(\Q_2) \intersect \overline{J(\Q)} = \{\infty\}$
is positive.
\end{theorem}

\begin{proof}
  Proposition~\ref{P:locally constant image of rholog}
  yields $U \subseteq \calF_g(\Z_p)\setminus Z$ 
  containing the curve of Lemma~\ref{L:rholog of size 1}
  and satisfying the hypothesis of 
  Proposition~\ref{P:general density result}.
By Proposition~\ref{P:general density result}, 
$C(\Q) = C(\Q_2) \intersect \overline{J(\Q)} = \{\infty\}$
for $C \in \calF_g \intersect U$ outside a subset of relative upper density 
  at most $2 \cdot 2^{1-g} < 1$.
\end{proof}

\begin{remark}
The density of~$U$ will be rather small, so the lower bound on the lower
density of curves with just one rational point we obtain in this way
will also be very small. 
\end{remark}

\begin{remark}
\label{R:3-Selmer}
Although Theorem~\ref{T:positive density} says nothing for $g=2$, 
$\Eq_2(3)$ would imply that the lower density
of the set of $C \in \calF_2$ satisfying $C(\Q)=\{\infty\}$ is positive.
This implication can be proved by a similar argument,
using curves $3$-adically close to $y^2 = x^5 - x^3 - 1$.
\end{remark}

\subsection{Density tending to 1}

\begin{theorem} \label{T:main}
Fix $g > 1$. 
Then the lower density of the set of curves $C \in \calF_g$ 
satisfying $C(\Q) = C(\Q_2) \intersect \overline{J(\Q)} = \{\infty\}$
is at least $1 - (12g + 20) 2^{-g}$.
\end{theorem}

\begin{proof}
Apply Proposition~\ref{P:general density result}
to each congruence class in Remark~\ref{R:total measure 1},
and sum the results by using Corollary~\ref{C:average image of rholog}
for $p=2$:
the upper density of curves \emph{not} satisfying the condition is
at most $\left(1 + (6g+9) \right) 2^{1-g} = (12g+20) 2^{-g}$.
\end{proof}

\begin{remark}
We have $1 - (12g + 20) 2^{-g} > 0$ if and only if $g \ge 7$.
Also, $1 - (12g + 20) 2^{-g} \to 1$ as $g \to \infty$.
Using the refinement given in Remark~\ref{R:better rholog bound},
the bound could be improved to $1 - (6g + 11) 2^{-g}$,
which is positive also for $g = 6$.
\end{remark}

\begin{theorem} \label{T:dens_odd}
  Fix $g>1$ and an odd prime $p$.
  Assume $\Eq_g(p)$. 
Then the lower density of the set of curves $C \in \calF_g$ 
satisfying $C(\Q) = \{\infty\}$
is at least
  \[ 
     1 - \Bigl(
  1 + (p+1)^2 + \dfrac{p^2-p}{p-2}(2g-2)
    \Bigr) p^{1-g}.
  \]
\end{theorem}

\begin{proof}
Repeat the proof of Theorem~\ref{T:main}, using the bound for
  odd~$p$ in Proposition~\ref{P:image of general C}.
\end{proof}

\begin{corollary}
Fix $g \ge 4$.
Assume that $\Eq_g(p)$ holds for arbitrarily large primes~$p$.
Then the set of curves~$C \in \calF_g$ 
satisfying $C(\Q) = \{\infty\}$ has density~$1$.
\end{corollary}

\begin{proof}
The lower bound in Theorem~\ref{T:dens_odd} tends to $1$ as $p \to \infty$.
\end{proof}

\begin{remark}
For $g=3$, the lower bound tends to $0$ from below,
but using Remark~\ref{R:better rholog bound}
to cut the subtracted term essentially in half,
we would obtain a limit of $1/2$.
Thus if $\Eq_3(p)$ holds for arbitrarily large primes~$p$,
then the set of curves~$C \in \calF_3$
satisfying $C(\Q) = \{\infty\}$
has lower density at least~$1/2$.
\end{remark}

\begin{remark}
\label{R:100 percent}
Several authors have presented heuristics or conditional proofs
that suggest that in an algebraic family of curves of genus greater than~$1$,
the density of those having rational points other than points that
exist generically is $0$: 
see \cite{Poonen-Voloch2004}*{Conjecture~2.2}, 
\cite{Poonen2006-heuristic}, 
\cite{Granville2007}*{Conjecture~1.3(ii)}, 
and \cite{Stoll2009-preprint}*{Conjecture~1}.
\end{remark}


\subsection{Effectivity} \label{S:Eff}

Let $\calF_g^{\good}$ be the set of $C \in \calF_g$
such that $\sigma$ is injective,
the images $\rholog(C(\Q_2))$ and $\PP\sigma(\Sel_2 J)$ are disjoint,
and $C(\Q_2)$ contains no nontrivial torsion point of odd order.
Our proof of Theorem \ref{T:positive density} (resp., Theorem~\ref{T:main})
can be summarized as follows:
\begin{itemize}
\item if $C \in \calF_g^{\good}$, then Chabauty's method at the prime $2$
proves that $C(\Q)=\{\infty\}$;
\item $\calF_g^{\good}$ has positive lower density if $g \ge 3$
(resp., density at least $1 - (12g + 20) 2^{-g}$ for each $g>1$).
\end{itemize}

\begin{theorem}
\label{T:algorithm}
There is an algorithm that takes as input
an integer $g>1$ and a curve $C \in \calF_g$
and decides whether or not $C \in \calF_g^{\good}$.
\end{theorem}

\begin{proof}
One such algorithm (not an especially efficient one) proceeds as follows.
Let $\Q_2'$ be the subfield of $\Q_2$
consisting of elements algebraic over $\Q$.
The advantage of $\Q_2'$ over $\Q_2$ from the algorithmic point of view
is that an element of $\Q_2'$ can be specified exactly with a finite
amount of data.
On the other hand, $\Q_2'$ approximates $\Q_2$ well
in the sense that $C(\Q_2')$ is dense in $C(\Q_2)$
and $J(\Q_2')$ is dense in $J(\Q_2)$.

First, use \cite{Poonen2001-torsion} to compute
the finite set $C(\Qbar) \intersect J(\Qbar)_{\tors}$.
Check each of its points for membership in $C(\Q_2)$ 
to compute $\calT \colonequals C(\Q_2) \intersect J(\Q_2)_{\tors}$.
Since $J(\Q_2)_{\tors} \subseteq J(\Q_2')$, we have $\calT \subseteq C(\Q_2')$.
If any $T \in \calT$ is of odd order greater than $1$, then return ``no''.

Next, use \cite{Stoll2001} to
compute the group $\Sel_2 J$, the group $J(\Q_2)/2J(\Q_2)$, 
and the map between them.
Compute $J(\Q_2)_{\tors}$ and its image in $J(\Q_2)/2J(\Q_2)$.
If the map 
$\sigma \colon \Sel_2 J \to J(\Q_2)/(2J(\Q_2)+J(\Q_2)_{\tors}) \isom \F_2^g$
is not injective,
then return ``no''.
Otherwise compute $\PP\sigma(\Sel_2 J)$.

Finally, we need to compute $\rholog(C(\Q_2))$.
For any $P \in C(\Qbar_2)$ and $\omega \in \HH^0(C,\Omega^1)$,
one can define the $2$-adic integral $\int_{\infty}^P \omega \in \Qbar_2$
(see \cite{McCallum-Poonen2012}*{Section~5.1}, for example).
For a divisor $D = \sum n_P P \in \Div C_{\Qbar_2}$, 
define $\int^D \omega \colonequals \sum n_P \int_{\infty}^P \omega$.
Given $D \in \Div C_{\Q'_2} \injects \Div C_{\Qbar_2}$ and $\omega$ as above,
the integral $\int^D \omega \in \Q_2$ can be computed
to any desired precision
by integrating formal power series and using the group law on $J$
(see~\cite{McCallum-Poonen2012}*{Section~8.3}, for example).

Choose degree~$0$ divisors $D_1,\ldots,D_g$ on $C_{\Q'_2}$
representing an $\F_2$-basis for $J(\Q_2)/(2J(\Q_2)+J(\Q_2)_{\tors})$.
Choose any $\Q$-basis $\omega'_1,\ldots,\omega'_g$
for $\HH^0(C,\Omega^1)$, say $\omega'_j \colonequals x^{j-1}\,dx/y$.
Compute the integrals $\int^{D_i} \omega'_j \in \Q_2$ to sufficient precision
that we can find a new basis $\omega_1,\ldots,\omega_g$ of $\HH^0(C,\Omega^1)$
guaranteed to make the matrix $(\int^{D_i} \omega_j)$ lie in $\GL_g(\Z_2)$.
This new basis defines a homomorphism $\log$ as in Section~\ref{S:logarithm}.

For each $T \in \calT$,
choose a uniformizer $t$ on $C_{\Q_2'}$ at $T$.
For each sufficiently small value of $t$ in $\Q_2$,
let $P_t$ be the corresponding point of $C(\Q_2)$ near $T$.
Each coordinate of $P_t$ is a power series in $\Q_2'[[t]]$
that can be calculated in the sense 
that any desired coefficient can be calculated.
The same is true for the power series 
$\ell_i(t) \colonequals \int_\infty^{P_t} \omega_i$
for each $i$.
Each $\ell_i(t)$ vanishes at $t=0$ since $T$ is torsion,
but some $\omega_i$ is nonvanishing at $T$
so the $\ell_i(t)$ do not all vanish to order~$2$.
Using Hensel's lemma, we can control the rate of convergence 
of all these power series
in order to compute an explicit open and closed neighborhood of $0$ in $\Q_2$
on which $\rho(t^{-1} \ell_1(t),\ldots,t^{-1} \ell_g(t))$
converges and is constant.
Then $\rholog$ is constant on 
the corresponding explicit neighborhood of $T$ in $C(\Q_2)$.
Such a neighborhood can be specified explicitly as a fiber
of the map $C(\Q_2) = \calC(\Z_2) \to \calC(\Z/2^e\Z)$ for some $e$,
where $\calC$ is an explicit proper $\Z_2$-model of $C$,
say the Weierstrass model.
Let $U$ be the union of these neighborhoods as $T$ varies;
thus we know $\rholog(U)$.

On $C(\Q_2)\setminus U$, $\log$ is bounded away from $\boldzero$.
Therefore, for any $P \in C(\Q_2) \setminus U$,
the value $\rholog(P)$ can be computed,
and the computation examines only finitely many $2$-adic digits
of the coefficients of $P$;
in other words, the computation succeeds with the same result
for all points in a fiber of $C(\Q_2) = \calC(\Z_2) \to \calC(\Z/2^e\Z)$ 
for some $e$.
Since $C(\Q_2) \setminus U$ is compact, it can be covered 
by finitely many such fibers.
For $e=1,2,\ldots$ in turn,
attempt to calculate $\rholog(P)$ for one point $P$ in each nonempty fiber
of $C(\Q_2)\setminus U \to \calC(\Z/2^e\Z)$
using only the precision specified by the image of $P$ in 
$\calC(\Z/2^e\Z)$.
We may fail for the first few $e$,
but the compactness argument guarantees that eventually
a successful $e$ will be found,
and then we know $\rholog(C(\Q_2)\setminus U)$.

Taking the union of $\rholog(U)$ and $\rholog(C(\Q_2)\setminus U)$
yields $\rholog(C(\Q_2))$.
Return ``yes'' or ``no''
according to whether 
$\rholog(C(\Q_2))$ and $\PP\sigma(\Sel_2 J)$ are disjoint.
\end{proof}

\begin{corollary}
\label{C:effective}
There is an algorithm based on Chabauty's method at the prime $2$ 
that succeeds in determining $C(\Q)$ 
for a computable set of curves $C \in \calF_g$
of lower density at least the bound in Theorem~\ref{T:positive density}
or Theorem~\ref{T:main} for any $g \ge 3$.
\end{corollary}


\section*{Acknowledgements} 

The idea to combine Chabauty's method 
with results on the average $2$-Selmer group size
is due to Manjul Bhargava and Benedict Gross,
whom we thank for discussions.
In particular, we thank Gross for lingering after giving a lecture
to explain the details of their argument to us; 
it was only after this 
that we had the idea that an equidistribution result (also proved 
by Bhargava and Gross)
could be used to refine a $2$-adic Chabauty argument 
enough to reduce the upper bound on $\#C(\Q)$ to $1$.
We thank Alice Guionnet and Scott Sheffield for discussions
regarding Lemma~\ref{L:E X_n},
and Jennifer Park for a discussion 
regarding Lemma~\ref{L:rholog of size 1}.
The first author thanks the Centre Interfacultaire Bernoulli at EPFL
for its hospitality during the period when most of the research
for this paper was done; 
his research was partially supported by the Guggenheim Foundation
and National Science Foundation grant DMS-1069236.
The second author thanks the German Science Foundation for supporting
his research through a grant within the framework of the DFG Priority Programme~1489
\emph{Algorithmic and Experimental Methods in Algebra, Geometry and Number Theory}.
He thanks ICERM at Brown University
for its hospitality during part of the semester program on
\emph{Complex and arithmetic dynamics}.

We also thank the referees for useful comments.

\end{document}